\documentclass{scrartcl}
\usepackage{a4wide}
\usepackage{amsmath}
\usepackage{lscape}
\usepackage{amsthm}
\usepackage[mdbch]{mathdesign}
\input xy
\xyoption{all}
\theoremstyle{definition}
\newtheorem{definition}{Definition}[section]
\newtheorem{example}[definition]{Example}
\newtheorem{remark}[definition]{Remark}
\theoremstyle{plain}
\newtheorem{theorem}[definition]{Theorem}
\newtheorem{proposition}[definition]{Proposition}
\newtheorem{corollary}[definition]{Corollary}
\newtheorem{lemma}[definition]{Lemma}

\newcommand{\al}{\alpha}
\newcommand{\ep}{\varepsilon}
   
 \newcommand{\la}{\lambda}

\newcommand{\si}{\sigma}
\newcommand{\ph}{\varphi}
\newcommand{\om}{\omega}
\newcommand{\Om}{\Omega}

\newcommand{\mf}{\mathfrak{m}}\newcommand{\nf}{\mathfrak{n}}
\newcommand{\C}{\mathbb{C}}\newcommand{\Z}{\mathbb{Z}}
\newcommand{\K}{\mathbb{K}}
\newcommand{\Max}{\mathrm{Max}}\newcommand{\Supp}{\mathrm{Supp}}

\newcommand{\shrp}{\dagger}

\DeclareMathOperator{\Hom}{Hom}
\DeclareMathOperator{\Aut}{Aut}\DeclareMathOperator{\Id}{Id}
\DeclareMathOperator{\sgn}{sgn}

\title{Pseudo-Unitarizable Weight Modules over Generalized Weyl Algebras}
\author{Jonas T. Hartwig}

\date{}

\begin{document}
\maketitle

\begin{abstract}
We define a notion of pseudo-unitarizability for weight modules over a
generalized Weyl algebra (of rank one, with commutative coeffiecient
ring $R$), which is
assumed to carry an involution of the form $X^*=Y$, $R^*\subseteq R$.
We prove that a weight module $V$ is pseudo-unitarizable
iff it is isomorphic to its finitistic dual $V^\sharp$.
Using the classification of weight modules by Drozd, Guzner and Ovsienko,
we obtain necessary and sufficient conditions for an indecomposable
weight module to be isomorphic to its finitistic dual, and thus to
be pseudo-unitarizable. Some examples are given,
including $U_q(\mathfrak{sl}_2)$ for $q$ a root of unity.
\end{abstract}

\section{Introduction}
For a $\ast$-algebra $A$ over $\C$ and an $A$-module $V$,
a basic question is whether $V$ is unitarizable. That is,
can $V$ be equipped with a positive definite inner product which is
$A$-admissible, i.e. $(av,w)=(v,a^\ast w)$ for $a\in A,v,w\in V$?
This is so in many well-behaved examples, like simple finite-dimensional
modules over a finite-dimensional group algebra,
but unfortunately false in general.
However, the modules for which this is false
might still be pseudo-unitarizable
in the 
 sense of having an admissible inner product
which is non-degenerate but not necessarily positive definite.

A new feature for this broadened notion
is that there may exist pseudo-unitarizable indecomposable modules
which are not simple.

Such indefinite inner product spaces have been thoroughly
studied in the analytical setting of operator algebras,
see \cite{KS}.
There are also many applications to areas in physics,
for example quantum field theory. See \cite{MS} and
references therein.

On the algebraic side,
existence and uniqueness questions of such indefinite inner products
was considered in \cite{MT} in the general situation of
$A$ being a $\ast$-algebra over an algebraically closed field
and $M$ being a finite-dimensional $A$-module,
or a weight $A$-module with finite-dimensional weight spaces.
Among other things, it was shown that
an $A$-module $M$ has a non-degenerate admissible form iff
$M$ is isomorphic to its finitistic dual $M^\sharp$.
In \cite{MT2} the authors described all simple weight (with respect to a
Cartan subalgebra) modules with
finite-dimensional weight spaces
over a complex finite-dimensional semisimple Lie algebra
which are pseudo-unitarizable with a non-degenerate symmetric form.

In this paper we consider generalized Weyl algebras (GWAs).
These are certain noncommutative rings,
first introduced in \cite{B}, and studied since in many different papers
(see \cite{BB}, \cite{BvO}, \cite{BL} and references therein).
The class contains a wide range of examples such as
ambiskew polynomial rings \cite{Jor},
which includes Noetherian generalized down-up algebras \cite{CS};
$U(\mathfrak{sl}_2)$ and its various
deformations and generalizations (see for example \cite{BvO})
as well as the first Weyl algebra and quantum Weyl algebra.
Unitarizable modules over GWAs (and, more generally, twisted GWAs)
with ground field $\C$ were studied in \cite{MazTur2002}.
In particular simple unitarizable weight modules were classified.

We will consider GWAs of rank one, $A=R(\si,t)$, and assume that $R$ is a
commutative ring.
One of the problems with GWAs in this generality is that there is no canonical
choice of a ground field.
For such GWAs, all indecomposable weight
modules with finite-dimensional weight spaces
were classified in \cite{DGO},
up to indecomposable elements in a skew polynomial ring over a field.
There are five families of modules,
some of them depending on many parameters.
It is interesting, therefore, to ask
if some of these modules possess extra structure.

The purpose of this paper is two-fold:
\begin{enumerate}
\item[1)] To define an appropriate notion of pseudo-unitarizability for weight modules
over a generalized Weyl algebra equipped with an involution
satisfying $X^\ast=Y$, $Y^\ast=X$, $R^\ast\subseteq R$.
See Definition \ref{dfn:pu}. 

\item[2)]
To find conditions on the parameters
of the indecomposable weight modules $V$ over a
generalized Weyl algebra, which
are necessary and sufficient for the
modules to be pseudo-unitarizable. The main results here are
Theorems \ref{thm:inforbnobr}, \ref{thm:inforbwbr}, \ref{thm:Vomfdual},
\ref{thm:firstkind}, and \ref{thm:2ndkind}
which completely answers this question in the case of real orbit $\om$,
i.e. $\mf^\ast=\mf\;\forall\mf\in\om$.
\end{enumerate}

After recalling some basic definitions in Section \ref{sec:setup},
we give in Section \ref{sec:a}
the definition of admissible form
and of the finitistic dual $V^\sharp$.
We prove analogs of some results from \cite{MT}
such as Proposition \ref{prop:bij} on the correspondence
between forms and morphisms.

In Section \ref{sec:DGOclass} we recall the
classification theorem from \cite{DGO}.
We have collected all notation necessary in Section \ref{sec:notation}.

In Section \ref{sec:description} we consider in turn each type 
of indecomposable weight module and give necessary and sufficient
conditions for the existence of a non-degenerate admissible form.

We end by considering some examples in Section \ref{sec:examples}.
In particular we obtain in Section \ref{sec:ex_uqsl2} conditions for
indecomposable non-simple modules over $U_q(\mathfrak{sl}_2)$
($q$ a root of unity), to have non-degenerate admissible forms.

\section{Setup} \label{sec:setup}
Let
\begin{itemize}
\item $R$ be a commutative ring with $1$,
\item $\ast:R\to R$ an automorphism of order $1$ or $2$,
\item $\si:R\to R$ an automorphism commuting with $\ast$, and
\item $t\in R$ be selfadjoint, i.e. $t^\ast=t$.
\end{itemize}
Let $A=R(\si,t)$ be the associated \emph{generalized Weyl algebra} (GWA) \cite{B}.
Thus $A$ is the ring generated by the set $R\cup\{X,Y\}$, where $X,Y$ are two
new symbols,
with the relations that $R$ is a subring of $A$ and
\begin{equation}\label{eq:gwarels}
YX=t,\quad XY=\si(t),\quad Xr=\si(r)X,\quad Yr=\si^{-1}(r)Y\quad \forall r\in R.
\end{equation}
By \eqref{eq:gwarels}, $\ast$ extends to an involution on $A$
(i.e. $(a+b)^\ast=a^\ast+b^\ast, (ab)^*=b^*a^*$, $a^{\ast\ast}=a$, $\forall a,b\in A$)
by requiring
\[X^\ast=Y,\qquad Y^\ast=X.\]
Relations \eqref{eq:gwarels} also imply that $A$ is a $\Z$-graded ring $A=\oplus_{n\in\Z} A_n$
with gradation given by $\deg X=1, \deg Y=-1, \deg r=0\; \forall r\in R$.
Let $\Om$ be the set of orbits for the action of $\si$ on the set $\Max(R)$
of maximal ideals of $R$. For $\om\in\Om$ we let $R_\om$ denote the direct sum
of all the $R$-modules $R/\mf$ for $\mf\in\om$:
\begin{equation}
R_\om=\bigoplus_{\mf\in\om} R/\mf.
\end{equation}
The $R$-module $R_\om$ will be used as a subtitute for a ground field, when defining
admissible forms in Section \ref{sec:admiss}.
The automorphism $\si$ induces isomorphisms $R/\mf\to R/\si(\mf)$, $\mf\in\Max(R)$,
which we also denote by $\si$. Extending additively, we get a map
$\si:R_\om\to R_\om$.
The automorphism $\ast$ of $R$ induces a map $R/\mf\to R/\mf^\ast$, and hence
a map $R_\om\to R_{\om^\ast}$ which will be called \emph{conjugation} and denoted $\la\mapsto \overline{\la}$.

\begin{remark}
Let $A=R(\si,t)$ be a GWA and $\ast$ an anti-involution on $A$ satisfying $R^\ast\subseteq R$ and
$X^\ast=\ep Y$, where $\ep\in R$ is invertible. Then, after a change of
generators, we can assume $\ep=1$ and thus that $t^\ast=t$.
Indeed, set $X_1=X$, $Y_1=\ep Y$ and $t_1=Y_1X_1=\ep t$.
Then $X_1Y_1=X\ep Y=\si(\ep)\si(t)=\si(t_1)$. Clearly $X_1r=\si(r)X_1$ and
$Y_1r=\si^{-1}(r)Y_1$, $\forall r\in R$. Moreover $X_1^\ast=Y_1$
so that $t_1^\ast=t_1$.
\end{remark}

\begin{definition}
A module $V$ over a ring, which contains $R$ as a subring, will be called a
\emph{weight module} if
$V=\oplus_{\mf\in\Max(R)}V_\mf$, where $V_\mf=\{v\in V:\mf v=0\}$. The $R$-submodules
$V_\mf$ of $V$ are
called \emph{weight spaces} and elements of $V_\mf$ are \emph{weight vectors of weight $\mf$}.
The \emph{support of $V$}, denoted $\Supp(V)$,
is defined as the set $\{\mf\in\Max(R):V_\mf\neq 0\}$.
\end{definition}

\section{Admissible forms and the finitistic dual} \label{sec:a}

\subsection{Motivation of definition}\label{sec:motiv}
In section \ref{sec:admiss} we will define an admissible form on a
weight $A$-module $V$ to be a certain biadditive form on $V$ with
values in the $R$-module $R_\om$. To motivate this definition,
let us first consider another, at first sight more natural, attempt at a definiton.

As we will see, a problem appears when $\om$ is finite.
Suppose therefore that $\om\in\Om$ is a finite orbit. Let $p=|\om|$.
Let $\om\in\Om$ and let $V$ be a weight module over $A$ with $\Supp(V)\subseteq\om$.
If we choose and fix an element $\mf\in\om$, we can define
a $R/\mf$-vector space structure on $V$ by $(r+\mf)v=\si^k(r)v$
if $v\in V_{\si^k(\mf)}$ and $0\le k<p$.
Then, for $v\in V_{\si^k(\mf)}$ and $\la=r+\mf\in R/\mf$,
\[X^p\la v= X^p \si^k(r) v=\si^{p+k}(r)X^p v=\si^p(\la)X^pv.\]
It would perhaps seem natural to define $V$ to be pseudo-unitarizable
if there is a nonzero \emph{admissible $R/\mf$-form} on $V$, i.e.
a map $G:V\times V\to R/\mf$ satisfying
\begin{subequations}
\begin{align}
\label{eq:Gform1}&\text{$G$ is additive in each argument},\\
\label{eq:Gform2}&G(\la v,w)=\la G(v,w)&&\text{for all $v,w\in V,\; \la\in R/\mf$,}\\
\label{eq:Gform3}&G(av,w)=G(v,a^\ast w)&&\text{for all $v,w\in V,\; a\in A$.}
\end{align}
\end{subequations}
However, then, for $v,w\in V$ and $\la\in R/\mf$,
\[G(X^p\la v,w)=G(\la v,Y^p w)=\la G(v,Y^p w)=\la G(X^p v,w),\]
while on the other hand, 
\[G(X^p\la v,w)=G(\si^p(\la)X^p v,w)=\si^p(\la)G(X^p v,w). \]
Thus, any weight module $V$ with $\Supp(V)\subseteq\om$ on which
$X^p\neq 0$ (or $Y^p\neq 0$ for analogous reasons)
would automatically be excluded from the possibility of being
pseudo-unitarizable (at least with a non-degenerate form),
unless $\si^p:R/\mf\to R/\mf$ is the
identity map for some (hence all) $\mf\in\om$.

Although $\si^p:R/\mf\to R/\mf$ is the identity in many important
examples (for example, if $R$ is a finitely generated algebra over
an algebraically closed field $k$ and
$\si$ is a $k$-algebra automorphism, then
$\si^p:R/\nf\to R/\nf$ is the identity
for any $\nf\in\Max(R)$ with $\si^p(\nf)=\nf$),
we feel that this notion of admissible form is too restrictive.

To remedy this situation we introduce in Section \ref{sec:admiss} a modified
definition of pseudo-unitarizability which has three advantages. First, 
no unnecessary restrictions applies as to which modules
can be pseudo-unitarizable when
$\si^p:R/\mf\to R/\mf$ is nontrivial.
 Secondly, the definition does not depend on any unnatural choice
of maximal ideal in the orbit. And thirdly, in the special case
when $\si^p:R/\mf \to R/\mf$ really is the identity map (and also
when the orbit $\om$ is infinite), the
definition is equivalent to the one above
in the sense that one form can be obtained
from the other in a bijective manner, as described in Proposition \ref{prop:fieldform2}.

\subsection{Admissible forms and pseudo-unitarizability}\label{sec:admiss}
Let $\om\in\Om$ and $V$ be a weight module over $A$ with $\Supp(V)\subseteq\om$.
\begin{definition}\label{dfn:admiss}
An \emph{admissible form} $F$ on $V$ is a map
 \[F:V\times V\to R_\om\]
such that
\begin{subequations}
\begin{align}
\label{eq:adm1nyy}
&\text{$F$ is additive in each argument,}\\
\label{eq:adm2nyy}
&F(rv,w)=rF(v,w) &&\text{for all $v,w\in V,\; r\in R$,}\\
\label{eq:adm3nyy}
&F(av,w)=\si^{\deg a}\big(F(v ,a^\ast w)\big) &&\text{for all $v,w\in V,\; a\in\cup_{n\in\Z}A_n$.}
\end{align}
\end{subequations}\label{eq:admnyy}
An admissible form $F$ is called \emph{non-degenerate} if for any nonzero $v\in V$ there
exist $w_1,w_2\in V$ such that $F(w_1,v)\neq 0\neq F(v,w_2)$. 
\end{definition}
\begin{definition}\label{dfn:pu}
A weight module $V$ over $A$, whose support is contained in an orbit,
is \emph{pseudo-unitarizable} if there exists a non-degenerate admissible form on $V$.
\end{definition}

Note that, since $\deg a^\ast=-\deg a$ for homogenous $a\in A$,
relation \eqref{eq:adm3nyy} is equivalent to
$F(v,aw)=\si^{\deg a}\big(F(a^\ast v,w)\big)$.

\subsection{Relation to admissible $R/\mf$-forms}
In view of the discussion in Section \ref{sec:motiv} we make
the following definition.
\begin{definition}\label{dfn:torsion_trivial}
We call $\om\in\Om$ \emph{torsion trivial} if
whenever $\mf\in\om$, $n\in\Z$ and $\si^n(\mf)=\mf$
then the induced map $\si^n:R/\mf\to R/\mf$ is the identity.
\end{definition}

Assume that $\om\in\Om$ is torsion trivial.
For $\mf_1,\mf_2\in\om$, say $\mf_2=\si^n(\mf_1)$,
define $\si_{\mf_1,\mf_2}=\si^n:R/\mf_1\to R/\mf_2$.
Then $\si_{\mf_1,\mf_2}$ is independent of the choice (if any)
of $n$, since $\om$ is torsion trivial.
Fix $\mf\in\om$. Let $V$ be a weight $A$-module with $\Supp(V)\subseteq\om$.
Give $V$ the structure of an $R/\mf$-vector space by
$(r+\mf)v = \si_{\mf,\si^k(\mf)}(r+\mf)v=\si^k(r) v$ for $v\in V_{\si^k(\mf)}$ and
$r+\mf\in R/\mf$.

\begin{proposition} \label{prop:fieldform2}
When $\om$ is torsion trivial,
there is a bijective correspondence between admissible
forms $F$ and admissible $R/\mf$-forms $G$ on $V$.
\end{proposition}
\begin{proof}
Given $F$, define $G$ by $G=\pi\circ F$, where $\pi:R_\om\to R/\mf$
is given by
\[\pi\big( (\la_\nf)_{\nf\in\om} \big) = \sum_{\nf\in\om}\si_{\nf,\mf}(\la_\nf).\]
Since $F$ is biadditive, so is $G$. To verify \eqref{eq:Gform2},
let $\nf=\si^k(\mf)\in\om$ be arbitrary, $v\in V_{\si^k(\mf)}$, $w\in V$
and $\la=r+\mf\in R/\mf$. Then,
using that $F(V_\nf,V)\subseteq R/\nf$, which follows from \eqref{eq:adm2nyy},
we have
\begin{align*}
G(\la v,w)&=\pi(F(\si^k(r)v,w))=\si^{-k}\big(\si^k(r)F(v,w)\big)
=r\si^{-k}\big(F(v,w)\big)=\\&=\la G(v,w). 
\end{align*}
To show \eqref{eq:Gform3}, let $\nf\in\om, v\in V_\nf, a\in A_k$.
Then $av\in V_{\si^k(\nf)}$ so
\begin{align*}
G(av,w)&=\si_{\si^k(\nf),\mf}\big(F(av,w)\big)=\si_{\si^k(\nf),\mf}\si^k\big( F(v,a^\ast w ) \big)
=\si_{\nf,\mf}\big( F(v,a^\ast w)\big)=\\
&=G(v,a^\ast w).
\end{align*}
This proves that $G$ is an admissible $R/\mf$-form on $V$.

Conversely, given $G$, define $F$ by
\[F(v,w)=\si_{\mf,\nf}\big(G(v,w))\quad\text{for $v\in V_\nf,\; w\in V$.}\]
Then $F$ is biadditive. To prove \eqref{eq:adm2nyy}, let
 $\nf=\si^k(\mf)\in\om, v\in V_{\nf}, w\in V$ and $r\in R$. Put
 $\la=r+\mf$. We have
\begin{align*}
F(\si^k(r)v,w)&=\si^k\big(G(\si^k(r)v,w)\big)=
\si^k\big(G(\la v,w) \big)=\si^k\big( \la G(v,w)\big)=\\
&=\si^k(r)\si^k\big(G(v,w)\big)=\si^k(r) F(v,w).
\end{align*}
Since $r$ was arbitrary, \eqref{eq:adm2nyy} is proved.
It remains to show that $F$ satisfies \eqref{eq:adm3nyy}.
Let $v\in V_\nf, a\in A_k$. Then
\[F(av,w)=\si_{\mf,\si^k(\nf)}\big(G(av,w)\big)
=\si^k\circ \si_{\mf,\nf}\big( G(v,a^\ast w)\big) = \si^k\big( F(v,a^\ast w)\big).\]
Thus $F$ is an admissible form on $V$.
\end{proof}

\subsection{Symmetric and real orbits} \label{sec:symorb}
\begin{definition}
An orbit $\om\in\Om$ is called \emph{symmetric} if $\mf^\ast\in\om$
for any $\mf\in\om$, and \emph{real} if $\mf^\ast=\mf$
for any $\mf\in\om$.
\end{definition}
\begin{proposition}
If $\om$ is symmetric but not real, then
$|\om |$ is finite, even, and $\mf^\ast=\si^{|\om |/2}(\mf)$
for any $\mf\in\om$.
\end{proposition}
\begin{proof}
Since $\om$ is symmetric but not real,
there is some $\nf\in\om$ such that $\nf\neq \nf^*=\si^N(\nf)$ for some $N\neq 0$.
Then
\[\nf=\nf^{\ast\ast}=\si^N(\nf)^\ast=\si^N(\nf^\ast)=\si^{2N}(\nf).\]
Hence $|\om |=p<\infty$ and $2N$ is a multiple of $p$.
Without loss of generality we can assume $0<N<p$. Then $2N=p$ is the only
possibility. Thus $\nf^\ast=\si^{|\om |/2}(\nf)$. Since any $\mf\in\om$
has the form $\si^k(\nf)$, and $\si$ and $\ast$ commute,
it follows that $\mf^\ast=\si^{|\om |/2}(\mf)$ for any $\mf\in\om$. 
\end{proof}

\subsection{Orthogonality of weight spaces} \label{sec:orthog}

\begin{proposition} \label{prop:orthog}
Let $\om\in\Om$ and let $V$ be a weight $A$-module with $\Supp(V)\subseteq\om$.
If $F$ is an admissible form on $V$, then $F(V_\mf,V_\nf)=0$ for any $\mf,\nf\in\om$
with $\mf\neq\nf^\ast$.
\end{proposition}
\begin{proof}
By \eqref{eq:adm2nyy} and \eqref{eq:adm3nyy},
\[
(\mf+\nf^\ast)F(V_\mf,V_\nf)=F(\mf V_\mf,V_\nf)+F(V_\mf,\nf V_\nf)=0.
\]
If $\mf\neq\nf^\ast$ then $\mf+\nf^\ast=R\ni 1$ so $F(V_\mf,V_\nf)=0$.
\end{proof}

\begin{corollary} \label{cor:orthog1}
Let $\om\in\Om$ be an orbit. If there exists a pseudo-unitarizable weight $A$-module $V$
with $\Supp(V)\subseteq\om$, then $\om$ is symmetric.
\end{corollary}
\begin{proof}
If $V$ is pseudo-unitarizable, it has a nonzero admissible form $F$. Since 
$F$ is nonzero
and $V$ is a weight module,
 $F(V_\mf,V_\nf)\neq 0$ for some $\mf,\nf\in\Supp(V)\subseteq\om$.
By \mbox{Proposition \ref{prop:orthog}},
 $\mf^\ast=\nf\in\om$. If $\mf_1\in\om$ is arbitrary, then $\mf_1=\si^n(\mf)$ for some
 $n$ and $\mf_1^\ast=\si^n(\mf)^\ast=\si^n(\mf^\ast)=\si^n(\nf)\in\om$.
This proves that $\om$ is symmetric.
\end{proof}
\begin{corollary} \label{cor:orthog2}
If $\om\in\Om$ is real and $V$ is a weight $A$-module with $\Supp(V)\subseteq\om$,
then the weight spaces of $V$ are pairwise orthogonal with respect to any
admissible form. 
\end{corollary}
\begin{proof}
This is immediate from Proposition \ref{prop:orthog}.
\end{proof}

\subsection{The finitistic dual $V^\sharp$}
Let $\om\in\Om$ and $V$ be a weight module over $A$ with $\Supp(V)\subseteq\om$.
Suppose $F$ is an admissible form on $V$. Let $u\in V$. Define
 $\tilde F_u:V\to R_\om$ by $\tilde F_u(v)=F(u,v)$.
\begin{proposition}\label{prop:tildeFu}
The map $\tilde F_u$ has the following properties:
\begin{subequations}
\begin{align}
\label{eq:tildeFu1}
& \tilde F_u(v_1+v_2)=\tilde F_u(v_1)+ \tilde F_u(v_2)&\forall v_1,v_2\in V,\\
\label{eq:tildeFu2}
& \tilde F_u(rv)=r^\ast \tilde F_u(v)&\forall r\in R, v\in V,\\
\label{eq:tildeFu3}
& \tilde F_u(V_\mf)=0\quad\text{for all but finitely many $\mf\in\om$.}
\end{align}
\end{subequations}
\end{proposition}
\begin{proof}
\eqref{eq:tildeFu1}, \eqref{eq:tildeFu2} follow from \eqref{eq:adm1nyy}-\eqref{eq:adm3nyy}.
For \eqref{eq:tildeFu3}, write $u=\sum_{i=1}^n u_i$, where $u_i\in V_{\mf_i}$.
Then if $\nf\in\om\backslash\{\mf_1^\ast,\ldots,\mf_n^\ast\}$ we get
\[\tilde F_u(V_\nf)=F(u_1,V_\nf)+\cdots+F(u_n,V_\nf)=0\]
by Proposition \ref{prop:orthog}.
\end{proof}

\begin{definition} Let $\om\in\Om$ and $V$ be a weight $A$-module with $\Supp(V)\subseteq\om$.
The \emph{finitistic dual} $V^\sharp$ of $V$ is the set of all maps
 $\ph:V\to R_\om$ satisfying the properties of Proposition \ref{prop:tildeFu}, i.e.
\begin{subequations}
\begin{align}
\label{eq:dual1}
& \ph (v_1+v_2)=\ph (v_1)+ \ph (v_2)&\forall v_1,v_2\in V,\\
\label{eq:dual2}
& \ph (rv)=r^\ast \ph (v)&\forall r\in R, v\in V,\\
\label{eq:dual3}
& \ph (V_\mf)=0\quad\text{for all but finitely many $\mf\in\om$.}
\end{align}
\end{subequations}
\end{definition}

\begin{proposition}
 $V^\sharp$ carries an $A$-module structure defined as follows.
Let $\ph\in V^\sharp$ and $r\in R$. Define $r\ph, X\ph, Y\ph:V\to R_\om$ by
\begin{subequations}
\begin{align}
\label{eq:dualmod1}
(r\ph)(v)&=\ph(r^\ast v)=r\ph(v),\\
\label{eq:dualmod2}
(X\ph)(v)&=\si\big(\ph(Yv)\big),\\
\label{eq:dualmod3}
(Y\ph)(v)&=\si^{-1}\big(\ph(Xv)\big),
\end{align}
\end{subequations}
for any $v\in V$.
\end{proposition}
\begin{proof}
First we must prove that $r\ph, X\ph, Y\ph\in V^\sharp$.
It is clear that $r\ph$ satisfies \eqref{eq:dual1},\eqref{eq:dual2},\eqref{eq:dual3}
since $\ph$ does.
Also $X\ph$ and $Y\ph$ satisfy \eqref{eq:dual1},\eqref{eq:dual3}.
We show \eqref{eq:dual2} for $X\ph$:
\begin{align*}
(X\ph)(rv)&\overset{\eqref{eq:dualmod2}}{=}\si\big(\ph(Yrv)\big)
=\si\big(\ph(\si^{-1}(r)Yv)\big)
\overset{\eqref{eq:dual2}}{=}\si\big(\si^{-1}(r)^\ast\big)\si\big(\ph(Yv)\big)=\\
&\overset{\eqref{eq:dualmod2}}{=}r^\ast(X\ph)(v).
\end{align*}
Analogously, $Y\ph$ satisfies \eqref{eq:dual2}.

We must also show that the relations in $A$ are preserved.
For any $\ph\in V^\sharp$ we have
\[
(YX\ph)(v)\overset{\eqref{eq:dualmod3}}{=}\si^{-1}\big((X\ph)(Xv)\big)
\overset{\eqref{eq:dualmod2}}{=}\ph(YXv)=\ph(tv)
\overset{\eqref{eq:dualmod1}}{=}(t\ph)(v) \quad\forall v\in V
\]
so $YX\ph=t\ph$.
Similarly, $XY\ph=\si(t)\ph$ for any $\ph\in V^\sharp$. Also, for
any $r\in R$ and $\ph\in V^\sharp$,
\begin{align*}
(Xr\ph)(v)&\overset{\eqref{eq:dualmod2}}{=}\si\big((r\ph)(Yv)\big)
\overset{\eqref{eq:dualmod1}}{=}\si\big(\ph(r^\ast Yv)\big)=
\si\big(\ph(Y\si(r^\ast) v)\big)=\\ &\overset{\eqref{eq:dualmod2}}{=}(X\ph)(\si(r)^\ast v)
\overset{\eqref{eq:dualmod1}}{=}\big(\si(r)X\ph)(v)\quad\forall v\in V.
\end{align*}
Analogously one proves that $Yr\ph=\si^{-1}(r)Y\ph$ for any $r\in R, \ph\in V^\sharp$.
Thus the relations of $A$ are preserved, so \eqref{eq:dualmod1}-\eqref{eq:dualmod3}
extends to an action of $A$ on $V^\sharp$.
\end{proof}

\begin{proposition}\label{prop:dualwtny}
 $V^\sharp$ is a weight $A$-module with
\begin{align}\label{eq:dualwtny}
(V^\sharp)_\mf&=
\big\{\ph\in V^\sharp\; :\;
\ph |_{V_\nf}=0 \text{ for all $\nf\in\om$ except possibly for $\nf=\mf^\ast$}\big\}\\
&=\big\{\ph\in V^\sharp\; :\; \ph(V)\subseteq R/\mf \big\}.
\end{align}
\end{proposition}
\begin{proof} Let $\ph\in V^\sharp$. Then
 $\mf \varphi = 0
\Leftrightarrow \varphi(\mf^\ast v)=0\;\; \forall v\in V
\Leftrightarrow \varphi |_{V_\nf}=0$ for all $\nf\in\om$ except possibly for $\nf=\mf^\ast$,
proving \eqref{eq:dualwtny}. The second equality holds since
 $\mf\ph=0 \Leftrightarrow \mf\ph(V)=0
\Leftrightarrow \ph(V)\subseteq ( R_\om)_\mf=R/\mf$. Since any $\varphi$ is
the sum of its projections $\varphi_\mf=\pi_\mf\circ\varphi$, where
$\pi_\mf:R_\om\to R/\mf$, $V^\sharp$ is a weight module.
\end{proof}

\begin{proposition}\label{prop:supp}
Let $\om\in\Om$ and let $V$ be a weight $A$-module with $\Supp(V)\subseteq\om$.
Then $\Supp(V^\sharp)=\Supp(V)^\ast=\big\{\mf^\ast\; :\; \mf\in \Supp(V)\big\}$.
\end{proposition}
\begin{proof}
Assume $\mf\in\Supp(V^\sharp)$ and let $0\neq\varphi\in (V^\sharp)_{\mf}$.
Then, by \eqref{eq:dualwtny}, $\varphi(v)\neq 0$ for some $v\in V_{\mf^\ast}$.
This implies that $\mf^\ast\in\Supp(V)$, i.e. $\mf\in\Supp(V)^\ast$.
Conversely, if $\mf\in\Supp(V)^\ast$ and
 $0\neq v\in V_{\mf^\ast}$ we can extend $v$ to an $R/\mf^\ast$-basis of $V_{\mf^\ast}$
and define $\ph\in V^\sharp$ by requiring that $\ph(V_\nf)=0$, $\nf\neq\mf^\ast$,
 $\ph(v)=1+\mf$ and $\ph(w)=0$ for all other basis vectors $w$ in $V_{\mf^\ast}$.
Then, by \eqref{eq:dualwtny}, $\ph\in (V^\sharp)_\mf$ so that $\mf\in\Supp(V^\sharp)$.
\end{proof}

\begin{proposition}\label{eq:duality}
If $\dim_{R/\mf} V_{\mf}<\infty$ for all $\mf\in\Supp(V)$ then the natural
inclusion of $V$ into $V^{\sharp\sharp}$ is an $A$-module isomorphism.
\end{proposition}
\begin{proof}
Define $\Psi:V\to V^{\sharp\sharp}$ by $\Psi(v)(\varphi)=\varphi(v)$ for $v\in V$,
 $\varphi\in V^{\sharp}$. Then
\[\Psi(Xv)(\ph)=\ph(Xv)\overset{\eqref{eq:dualmod3}}{=}\si\big((Y\ph)(v)\big)=
\si\big(\Psi(v)(Y\ph)\big)\overset{\eqref{eq:dualmod2}}{=}(X\Psi(v))(\ph)\]
for any $v\in V, \ph\in V^\sharp$. Similarly,
 $\Psi(Yv)=Y\Psi(v)$ and $\Psi(rv)=r\Psi(v)$ for any $r\in R$,
proving that $\Psi$ is an $A$-module homomorphism.
Let $v\in V$, $v\neq 0$ and write $v$ as a finite sum of weight vectors $v_\mf\neq 0$.
Then there exists $\varphi\in (V^\sharp)_{\mf^\ast}$ such that $\varphi(v)\neq 0$, i.e.
 $\Psi(v)(\varphi)\neq 0$ so $\Psi(v)\neq 0$. Thus $\Psi$ is injective. Also, by considering
dual bases, $\dim V_\mf=\dim (V^\sharp)_\mf$. Since $\Psi(V_\mf)\subseteq(V^{\sharp\sharp})_\mf$
we conclude that $\Psi$ is an isomorphism.
\end{proof}

Let $\om\in\Om$.
If $\Psi :V\to W$ is a homomorphism of weight $A$-modules with support in $\om$,
we define $\Psi^\sharp:W^\sharp\to V^\sharp$ by
\begin{equation}\label{eq:psisharp}
 \big(\Psi^\sharp(\ph)\big)(v)=\ph\big(\Psi(v)\big)\quad\forall v\in V,\forall \ph\in W^\sharp
 \end{equation}
\begin{proposition}
 $\Psi^\sharp$ is also an $A$-module homomorhpism. Moreover, $\sharp$ is a contravariant
 endofunctor on the category of weight $A$-modules with support in $\om$ or $\om^*$.
\end{proposition}
\begin{proof}
For any $v\in V, \ph\in W^\sharp, r\in R$, we have
\begin{align*}
\big(\Psi^\sharp(r\ph)\big)(v)&=(r\ph)\big(\Psi(v)\big)&\text{by definition of $\Psi^\sharp$}\\
&=\ph\big(r^\ast\Psi(v)\big)&\text{by $A$-module structure on $W^\sharp$}\\
&=\ph\big(\Psi(r^\ast v)\big)&\text{since $\Psi$ is an $A$-module morphism}\\
&=\big(\Psi^\sharp(\ph)\big)(r^\ast v)&\text{by definition of $\Psi^\sharp$}\\
&=\big( r\Psi^\sharp(\ph)\big)(v)&\text{by $A$-module structure on $V^\sharp$}
\end{align*}
In the same way one shows
that $\Psi^\sharp$ commutes with the actions of $X$ and $Y$. That $\sharp$ is
a functor is easy to check.
\end{proof}

\subsection{The bijection between forms and morphisms} 
Let $\om\in\Om$ and $V$ be a weight $A$-module with $\Supp(V)\subseteq\om$.
Assume $F$ is an admissible form on $V$. For $u\in V$,
recall that $\tilde F_u\in V^\sharp$ by Proposition \ref{prop:tildeFu}.
\begin{proposition}
The map $\tilde F :V\to V^\sharp$ defined by $u\mapsto \tilde F_u$ is
an $A$-module homomorphism.
\end{proposition}
\begin{proof} For any $r\in R, u,v\in V$ we have
\[
\tilde F_{ru}(v)= F(ru,v)=F(u,r^\ast v)=\tilde F_u(r^\ast v)=(r\tilde F_u)(v)\]
and
\[\tilde F_{Xu}(v)=F(Xu,v)=\si\big(F(u,Yv)\big)=\si\big(\tilde F_u(Yv)\big)=(X\tilde F_u)(v).\]
Similarly, $\tilde F_{Yu}=Y\tilde F_u$ for any $u\in V$.
Thus $\tilde F$ is an $A$-module homomorphism.
\end{proof}
The following proposition is analogous the corresponding result
proved in \cite{MT} for finite-dimensional modules over algebras.
\begin{proposition} \label{prop:bij}
The map $F\mapsto\tilde F$ is an isomorphism of abelian groups between
the space of admissible
forms on $V$ and $\Hom_A(V,V^\sharp)$. Moreover, non-degenerate forms correspond to
isomorphisms.
\end{proposition}
\begin{proof}
Given $\Phi\in\Hom_A(V,V^\sharp)$, define $\hat\Phi:V\times V\to R$ by
 $\hat\Phi(v,w)=\Phi(v)(w)$. Then $\hat\Phi$ is an admissible form on $V$ and the maps
 $F\mapsto\tilde F$ and $\Phi\mapsto\hat\Phi$ are inverses to each other. If
 $\hat\Phi(v,w)=0\; \forall w$ implies that $v=0$, then $\Phi$ is injective. If
 $\hat\Phi(v,w)=0\; \forall v$ implies that $w=0$, then $\Phi$ is surjective.
This proves the last claim.
\end{proof}

\subsection{A semi-simplicity condition}
\begin{proposition}\label{prop:semisimple}
Let $V$ be a weight $A$-module, with $\Supp(V)$
contained in a real orbit, such that
 $\dim_{R/\mf} V_\mf=1\;\forall\mf\in\Supp(V)$.
If $V^\sharp\simeq V$ then $V$ is semi-simple.
\end{proposition}
\begin{proof}
If $V^\sharp\simeq V$, then, by Proposition \ref{prop:bij}, $V$ has
a non-degenerate admissible form $F$. Let $U$ be any submodule of
$V$.
Then $U$ is itself a weight module and, since $\dim_{R/\mf}V_\mf=1$
for all $\mf\in\Supp(V)$, we have $U=\oplus_{\mf\in S} V_\mf$
for some subset $S\subseteq\Supp(V)$.
Let $U^\bot=\{v\in V\; : \; F(v,u)=0 \;\forall u\in U\}$.
By the defining properties of an admissible form
\eqref{eq:admnyy}, $U^\bot$ is an $A$-submodule of $V$.
On the other hand, by Corollary \ref{cor:orthog2} and the non-degeneracy
of $F$, we have $F(V_\mf,V_\nf)=0$ iff $\mf\neq\nf$ for $\mf,\nf\in\Supp(V)$.
Thus $U^\bot=\oplus_{\mf\in\Supp(V)\backslash S} V_\mf$.
This proves that $U\oplus U^\bot=V$.
Hence, any submodule has an invariant complement so $V$ is semi-simple.
\end{proof}

\subsection{Symmetric forms}
Recall that the map $R_\om\to R_{\om^\ast}$ induced by $\ast:R\to R$
is called conjugation and is denoted $\la\mapsto\overline{\la}$.
\begin{definition} Let $\om$ be a symmetric orbit and $F$ an admissible form
on a weight $A$-module $V$ with $\Supp(V)\subseteq\om$. The \emph{adjoint form}
 $F^\sharp:V\times V\to R_\om$ of $F$ is defined by
\begin{equation}
F^\sharp(v,w)=\overline{F(w,v)},\quad v,w\in V.
\end{equation}
It is easy to check that $F^\sharp$ is also an admissible form on $V$.
If $F=F^\sharp$, then $F$ is called \emph{symmetric}.
\end{definition}

If $\om$ is torsion trivial, we call an admissible $\K_\om$-form
$F$ symmetric if the corresponding admissible form is symmetric.

\begin{proposition}\label{prop:symmetric}
Suppose that $\om\in\Om$ is symmetric and torsion trivial.
Fix $\mf\in\om$ and put $\K_\om=R/\mf$.
Assume that conjugation on $\K_\om$ is non-trivial, and that the
fixed field under conjugation of $\K_\om$ is infinite, of characteristic not two.

Let $V$ be a finite-dimensional weight $A$-module with support in $\om$.
If $V$ has a non-degenerate admissible $\K_\om$-form,
then it has a symmetric non-degenerate admissible $\K_\om$-form.
\end{proposition}
The proof is exactly as in \cite{MT}, but we provide it for convenience.
\begin{proof}
Let $F:V\times V\to\K_\om$ be a non-degenerate admissible $\K_\om$-form
on $V$. Since conjugation is nontrivial, there is an $s\in\K_\om$
with $\overline{s}=-s$. Then $F_1=F+F^\sharp$ and $F_2=s(F-F^\sharp)$
are both symmetric admissible $\K_\om$-forms.
Define $f\in\K_\om[x]$ by $f(x)=\det(F_1'+xF_2')$. Here $F_i'$ denotes
the matrix of $F_i$ relative some $\K_\om$-linear basis of $V$.
Since $f(s^{-1})=\det(2F')\neq 0$, $f$ is a nonzero polynomial.
Among the infinitely many $r\in\K_\om$ with $\overline{r}=r$, pick one which is not a zero of $f$.
Then $F_1+rF_2$ is a symmetric non-degenerate admissible $\K_\om$-form on $V$.
\end{proof}

\begin{remark}\label{rem:MTresults}
Assume $R$ is a finitely generated algebra over an
algebraically closed field $\K$ of characteristic zero
and assume that $\si$ is a $\K$-automorphism of $R$.
Let $V$ be an indecomposable weight module over $A$ with support in a real orbit $\om$. 
Call two $\K$-forms $F_1,F_2$ on $V$ equivalent if there is an
automorphism $\varphi$ of $V$ and an element $\lambda\in\K, \la\neq 0$
such that $F_1(v,w)=\la F_2\big((\varphi(v),\varphi(w)\big)$ for all $v,w\in V$.

The following statements follow directly from Theorems 2,4 in \cite{MT}.

\begin{enumerate}
\item[1)]
If $V$ is simple and $V\simeq V^\sharp$, then there is a unique up to equivalence
non-degenerate admissible $\K$-form on $V$. If conjugation is nontrivial on $\K$
this form can be chosen to be symmetric, and if conjugation is trivial on $\K$,
the form can be chosen to be symmetric or skew-symmetric.
\item[2)] 
If there is a symmetric non-degenerate admissible $\K$-form on $V$,
then it is unique up to equivalence.
\end{enumerate}
\end{remark}

\section{The classification of weight modules}\label{sec:DGOclass}
In this section we review the classification of indecomposable weight
modules with finite-dimensional weight spaces
over a generalized Weyl algebra,
obtained by Drozd, Guzner and Ovsienko in \cite{DGO}.

\subsection{Notation} \label{sec:notation}
A maximal ideal $\mf$ of $R$ is called a \emph{break} if $t\in\mf$.
For $\om\in\Om$, let $B_\om$ be the set of all breaks in $\om$:
$B_\om=\{\mf\in\om: t\in\mf\}$. Often we put $p=|\om |$, $m=|B_\om |$.
Let $\K_\mf=R/\mf$. For $r\in R$ we define $r_\mf=r+\mf\in\K_\mf$.
For each $\om\in\Om$, fix an $\mf(\om)\in\om$ and put $\K_\om=\K_{\mf(\om)}$.

If $\om\in\Om$ is infinite, it is naturally ordered by
defining $\mf<\nf$ iff $\nf=\si^k(\mf)$ for some $k>0$.

If $|\om|=p<\infty$, define a ternary relation on $\om$ by $\mf<\mf'<\mf''$
if $\mf'=\si^i(\mf), \mf''=\si^k(\mf)$ for some $0<i<k<p$.
Let $m=|B_\om|$ and define a bijective corresponence $\Z_m\to B_\om$, $i\mapsto \mf_i$
such that $i<j<k$ in $\Z_\mf$ implies $\mf_i<\mf_j<\mf_k$ in $\om$
and $\mf_0=\mf(\om)$.
For $\mf\in\om$, let $j(\mf)$ denote the only $j\in\Z_m$ such that
$\mf_{j-1}<\mf\le \mf_j$.
Let $p_1,p_2,\ldots,p_m\in \Z_{>0}$ be minimal such that $\si^{p_j}(\mf_{j-1})=\mf_j$.
Equivalently, $p_i$ is the number of $\mf\in\om$ with $j(\mf)=i$.
Note that $p_1+p_2+\cdots +p_m=p$.
Furthermore, we put $\tau=\tau_\om=\si^p$.
Let $\K_\om[x,x^{-1};\tau]$ be the skew Laurent polynomial ring over $\K_\om$ with
automorphism $\tau$: $xa=\tau(a)x$ for $a\in\K_\om$.
Similarly, $\K_\om[x;\tau^{k}]$ is the skew polynomial ring over $\K_\om$
with automorphism $\tau^{k}$ ($k\in\Z_{\ge 0}$). An element
$f$ of such a skew (Laurent) polynomial ring $P$ is called \emph{indecomposable}
if the left $P$-module $P/Pf$ is indecomposable. Two elements $f,g\in P$
are called \emph{similar} if $P/Pf\simeq P/Pg$ as left $P$-modules.

Let $\mathbf{D}$ denote the free monoid on two letters $x,y$. Thus $\mathbf{D}$
is the set of words $w=z_1z_2\cdots z_n$, where $z_i\in\{x,y\}$, with
associative multiplication given by concatenation, and neutral element
being the empty word $\ep$ of zero length. A word $w$ is
an \emph{$m$-word} if its length $n$ is a multiple of $m\in\Z_{>0}$.
An $m$-word is \emph{non-periodic} if it is not a power of another
$m$-word.
We will let $\shrp:\mathbf{D}\to\mathbf{D}$, $w\mapsto w^\shrp$,
denote the automorphism given by
 $x^\shrp=y$, $y^\shrp=x$.
We also equip $\mathbf{D}$ with a $\Z$-action given by
\[1.z_1z_2\cdots z_n=z_2z_3\cdots z_n z_1.\] 
for $z_1z_2\cdots z_n\in\mathbf{D}$. Following \cite{DGO}, we
use the notation $w(k)$ for $k.w$.

When $\om$ is symmetric,
we will denote the map $\K_\om\to\K_\om$, which is induced
by the involution $*$ on $R$, by conjugation $a\mapsto \overline{a}$.

\subsection{The different kinds of modules}
\subsubsection{Infinite orbit without breaks}
Define $V(\om)$, where $\om\in\Om$, $|\om|=\infty$ and $B_\om=\emptyset$, as the space
 $V(\om)=\oplus_{\mf\in\om}\K_\mf$ with $A$-module structure given by
 $Xv=\si(t_\mf v)$ and $Yv=\si^{-1}(v)$ for $v\in \K_\mf$.

\subsubsection{Infinite orbit with breaks}
We use an alternative parametrization of these modules, which is more
convenient for our purposes.
It is easily seen to be equivalent to that of \cite{DGO}.
First we need some terminology.
Recall the order on infinite orbits $\om$ defined in Section \ref{sec:notation}.
An interval
$S$ in an infinite orbit $\om$ will be called \emph{supportive}
if it satisfies the following property:
if $S$ contains a minimal element $\nf_0$, then $\si^{-1}(\nf_0)\in B_\om$
and
if $S$ has a maximal element $\nf_1$, then $\nf_1\in B_\om$.
Let $I(S)$ be the set of \emph{inner breaks} of S:
\[I(S)=\{\mf\in S\cap B_\om: \si(\mf)\in S\}.\]
Now let $\om\in\Om$ be infinite with $B_\om\neq \emptyset$.
Let $S\subseteq\om$ be a supportive interval
and let $I_X$ be any subset of $I(S)$.
Define $V(\om,S,I_X)=\oplus_{\mf\in S} \K_\mf$ with,
for $v\in \K_\mf$,
\begin{align}
Xv&=
\begin{cases}
\si(t_\mf v),&\text{if $\mf\notin B_\om$},\\
\si(v),&\text{if $\mf\in I_X$},\\
0,&\text{otherwise},
\end{cases}&
Yv&=
\begin{cases}
\si^{-1}(v),&\text{if $\si^{-1}(\mf)\notin B_\om$},\\
\si^{-1}(v),&\text{if $\si^{-1}(\mf)\in I(S)\backslash I_X$},\\
0,&\text{otherwise}.
\end{cases}
\end{align}
Note that if $V=V(\om,S,I_X)$ then $S=\Supp(V)$ and $I_X=\{\mf\in I(S): XV_\mf\neq 0\}$.

\subsubsection{Finite orbit without breaks}
Given an orbit $\om$, with $|\om|=p<\infty$ and $B_\om=\emptyset$,
and an indecomposable polynomial 
 $f=\al_0+\al_1x+\cdots+a_dx^d\in \K_\om[x,x^{-1};\tau]$
with $\al_0\neq 0\neq\al_d$, define $V(\om,f)=\oplus_{\mf\in\om}(\K_\mf)^d$
with $A$-module structure given by defining for $v\in(\K_\mf)^d$

\begin{subequations}\label{eq:finorbnobr_both}
\begin{align}
\label{eq:finorbnobrX}
Xv&=
\begin{cases}
\si(t_\mf v),&\text{if $\mf\neq\mf(\om)$},\\
\si(F_ft_\mf v),&\text{if $\mf=\mf(\om)$},
\end{cases}\\
\label{eq:finorbnobrY}
Yv&=
\begin{cases}
\si^{-1}(v),&\text{if $\si^{-1}(\mf)\neq\mf(\om)$},\\
F_f^{-1}\si^{-1}(v),&\text{if $\si^{-1}(\mf)=\mf(\om)$},
\end{cases}
\end{align}
\end{subequations}
where
\[F_f=\begin{bmatrix}
0&0&0&\cdots&0&-\al_0/\al_d\\
1&0&0&\cdots&0&-\al_1/\al_d\\
0&1&0&\cdots&0&-\al_2/\al_d\\
\vdots&\vdots&\vdots&\ddots&\vdots&\vdots\\
0&0&0&\cdots&1&-\al_{d-1}/\al_d
\end{bmatrix}.
\] 

\subsubsection{Finite orbit with breaks, first kind}
Let $\om\in\Om$, $|\om|=p<\infty$ and $B_\om\neq\emptyset$.
Let $i\in\Z_m$ and $w=z_1z_2\cdots z_n\in\mathbf{D}$.
Consider $n+1$ symbols $e_0,e_1,\ldots,e_n$.
For $\mf\in\om$, let $V_\mf$ be the vector space over $\K_\mf$ with
basis consisting of all pairs $[\mf,e_k]$ such that
 $i+k=j(\mf)$ in $\Z_m$. Put $V(\om,i,w)=\oplus_{\mf\in\om}V_\mf$
and equip it with an $A$-module structure by

\begin{align}
\label{eq:FinOrbBr1X}
X[\mf,e_k]&=
\begin{cases}
\si(t_\mf) [\si(\mf),e_k],&\text{if $\mf\notin B_\om$},\\
[\si(\mf),e_{k+1}],&\text{if $\mf\in B_\om$ and $z_{k+1}=x$},\\
0,&\text{otherwise},
\end{cases}\\
\label{eq:FinOrbBr1Y}
Y[\mf,e_k]&=
\begin{cases}
[\si^{-1}(\mf),e_k],&\text{if $\si^{-1}(\mf)\notin B_\om$},\\
[\si^{-1}(\mf),e_{k-1}],&\text{if $\si^{-1}(\mf)\in B_\om$ and $z_k=y$},\\
0,&\text{otherwise}.
\end{cases}
\end{align}

\subsubsection{Finite orbit with breaks, second kind}
Define $V(\om,w,f)$, where $\om\in\Om$, $|\om|=p<\infty$ and
$|B_\om|=m>0$, $w=z_1z_2\cdots z_n\in\mathbf{D}\backslash \{\ep\}$ is a non-periodic $m$-word,
and $f=a_1+a_2x+\cdots+a_dx^{d-1}+x^d\neq x^d$ is an indecomposable
element of $\K_\om[x;\tau^{n/m}]$
 (it should be $\tau^{n/m}$ and not just $\tau$ as stated in \cite{DGO}), as follows.
Consider $dn$ symbols $e_{ks}$ ($k=1,\ldots,n$, $s=1,\ldots,d$).
For $\mf\in\om$, let $V_\mf$ be the vector space over $\K_\mf$
with basis consisting of all pairs $[\mf,e_{ks}]$ such that
 $k\equiv j(\mf)\;\; (\text{mod}\; m)$. Define $V(\om,w,f)=\oplus_{\mf\in\om} V_\mf$
and equip it with an $A$-module structure by

\begin{align}
\label{eq:2ndkindX}
X[\mf,e_{ks}]&=
\begin{cases}
\si(t_\mf) [\si(\mf),e_{ks}],&\text{if $\mf\notin B_\om$},\\
[\si(\mf),e_{k+1,s}],&\text{if $\mf\in B_\om$, $k<n$, $z_{k+1}=x$},\\
[\si(\mf),e_{1,s+1}],&\text{if $\mf\in B_\om$, $k=n$, $z_1=x$, $s<d$},\\
-\sum_{r=1}^d \si(a_r)[\si(\mf),e_{1r}],&\text{if $\mf\in B_\om$, $k=n$, $z_1=x$, $s=d$},\\
0,&\text{otherwise},
\end{cases}\\
\label{eq:2ndkindY}
Y[\mf,e_{ks}]&=
\begin{cases}
[\si^{-1}(\mf),e_{ks}],&\text{if $\si^{-1}(\mf)\notin B_\om$},\\
[\si^{-1}(\mf),e_{k-1,s}],&\text{if $\si^{-1}(\mf)\in B_\om$, $k>1$, $z_k=y$},\\
[\si^{-1}(\mf),e_{n,s-1}],&\text{if $\si^{-1}(\mf)\in B_\om$, $k=1$, $z_1=y$, $s>1$},\\
-\sum_{r=1}^d a_r^\circ [\si^{-1}(\mf ),e_{nr}],&\text{if $\si^{-1}(\mf)\in B_\om$,
 $k=1$, $z_1=y$, $s=1$},\\
0,&\text{otherwise}.
\end{cases}
\end{align}
Here $a_{d+1-r}^\circ=\tau^{r-1}(a_r)$, i.e. $a_r^\circ=\tau^{d-r}(a_{d+1-r})$.
As compared to \cite{DGO}, we changed notation
from $e_{ks}$ to $e_{k,d+1-s}$ in the case when $z_1=y$.

The weight diagram of
a module of the form $V=V(\om,w,f)$, where the first letter of
$w$ is $z_1=x$, is illustrated in Figure \ref{fig:Vtyp2}.
Each dot $\xymatrix@R-10pt{\bullet \save[]+<0pt,9pt>*{\mf}\restore}$
is a one-dimensional (over $R/\mf$) subspace of the weight space $V_\mf$.
Arrows going in the right direction correspond to $X$ while left arrows
correspond to $Y$.
The diagram 
$\xymatrix@R-10pt{
\bullet \save[]+<0pt,12pt>*{\mf}\restore \ar@/^/[r]
&\bullet \save[]+<0pt,12pt>*{\si(\mf)}\restore \ar@/^/[l]}$
means that $X$ and $Y$ act bijectively on the corresponding one-dimensional
subspaces.
We shall write
\[\xymatrix@R-10pt{
&\bullet \save[]+<0pt,12pt>*{\si(\mf)}\restore \ar@/^/[r]_{n}
&\bullet \save[]+<0pt,12pt>*{\si^n(\mf)}\restore \ar@/^/[l]
}
\]
to denote the weight diagram
\[
\xymatrix@R-10pt{
&\bullet \save[]+<0pt,12pt>*{\si(\mf)}\restore \ar@/^/[r]
&\bullet \save[]+<0pt,12pt>*{\si^2(\mf)}\restore \ar@/^/[l] \ar@/^/[r]
&\bullet \save[]+<0pt,12pt>*{ }\restore \ar@/^/[l] \ar@{{}*{\cdot\;}{}}[r]
&\bullet \save[]+<0pt,12pt>*{\si^{n-1}(\mf)}\restore            \ar@/^/[r]
&\bullet \save[]+<0pt,12pt>*{\si^n(\mf)}\restore \ar@/^/[l]
}.
\]
The diagram
$\xymatrix@R-10pt{
\bullet \save[]+<0pt,9pt>*{\mf}\restore \ar@{-}[r]^{z}
&\bullet \save[]+<0pt,9pt>*{\si(\mf)}\restore}$
where $z\in\{x,y\}$,
means that if $z=x$ then
$X$ acts bijectively from 
$\xymatrix@R-10pt{\bullet \save[]+<0pt,9pt>*{\mf}\restore}$
to
$\xymatrix@R-10pt{\bullet \save[]+<0pt,9pt>*{\si(\mf)}\restore}$
and $Y$ acts as zero on
$\xymatrix@R-10pt{\bullet \save[]+<0pt,9pt>*{\si(\mf)}\restore}$
while
if $z=y$, then
$Y$ is bijective as a map from 
$\xymatrix@R-10pt{\bullet \save[]+<0pt,9pt>*{\si(\mf)}\restore}$
to
$\xymatrix@R-10pt{\bullet \save[]+<0pt,9pt>*{\mf}\restore}$
and $X$ acts as zero on
$\xymatrix@R-10pt{\bullet \save[]+<0pt,9pt>*{\mf}\restore}$.
Often, in weight diagrams
each weight space is depicted as a column of dots.
In Figure \ref{fig:Vtyp2}, however, for clarity,
each column is only a subspace of a certain weight space,
and each weight is repeated $n/m$ times horizontally.
Recall that, by convention, $p_m=p_0$ and $\mf_m=\mf_0$.

\begin{landscape}
\begin{figure}
\caption{Weight diagram for $V(\om,w,f)$ when $z_1=x$}
\label{fig:Vtyp2}
\[\xymatrix@R-10pt@C-1pt{
&\bullet \save[]+<0pt,-12pt>*{ }+<0pt,24pt>*{\si(\mf_0)}\restore \ar@/^/[r]_{p_1}
&\bullet \save[]+<0pt,-12pt>*{e_{1,1}}+<0pt,24pt>*{\mf_1}\restore \ar@/^/[l] \ar@{-}[r]^{z_{2}}
&\bullet \save[]+<0pt,-12pt>*{ }+<0pt,24pt>*{\si(\mf_1)}\restore \ar@/^/[r]_{p_2}
&\bullet \save[]+<0pt,-12pt>*{e_{2,1}}+<0pt,24pt>*{\mf_2}\restore \ar@/^/[l]\ar@{-}[r]^{z_{3}}
&\bullet \save[]+<0pt,-12pt>*{ }+<0pt,24pt>*{ }\restore \ar@{{}*{\cdot\;}{}}[r] 
&\bullet \save[]+<0pt,-12pt>*{ }+<0pt,24pt>*{\si(\mf_{m-1}) }\restore \ar@/^/[r]_{p_m}
&\bullet \save[]+<0pt,-12pt>*{e_{m,1}}+<0pt,27pt>*{\mf_m }\restore \ar@/^/[l] \ar@{-}[r]^{z_{m+1}}
& \ar@{{}*{\cdot\;}{}}[rr]
& 
&\bullet \save[]+<0pt,-12pt>*{ }+<0pt,24pt>*{\si(\mf_0)}\restore \ar@/^/[r]_{p_1}
&\bullet \save[]+<0pt,-12pt>*{e_{n-m+1,1}}+<0pt,24pt>*{\mf_1}\restore \ar@/^/[l] \ar@{-}[r]^{z_{n-m+2}}
&\bullet \save[]+<0pt,-12pt>*{ }+<0pt,24pt>*{\si(\mf_1)}\restore \ar@/^/[r]_{p_2}
&\bullet \save[]+<0pt,-12pt>*{e_{n-m+2,1}}+<0pt,24pt>*{\mf_2}\restore \ar@/^/[l]\ar@{-}[r]^{z_{n-m+3}}
&\bullet \save[]+<0pt,-12pt>*{ }+<0pt,24pt>*{ }\restore \ar@{{}*{\cdot\;}{}}[r] 
&\bullet \save[]+<0pt,-12pt>*{ }+<0pt,24pt>*{\si(\mf_{m-1}) }\restore \ar@/^/[r]_{p_m}
&\bullet \save[]+<0pt,-12pt>*{e_{n,1}}+<0pt,27pt>*{\mf_m }\restore \ar@/^/[l] \ar@(d,u)[lllllllllllllllddd]_{z_1=x}
&\\ &\\ &\\
&\bullet \save[]+<0pt,-12pt>*{ }\restore \ar@/^/[r]_{p_1}
&\bullet \save[]+<0pt,-12pt>*{e_{1,2}}\restore \ar@/^/[l] \ar@{-}[r]^{z_{2}}
&\bullet \save[]+<0pt,-12pt>*{ }\restore \ar@/^/[r]_{p_2}
&\bullet \save[]+<0pt,-12pt>*{e_{2,2}}\restore \ar@/^/[l]\ar@{-}[r]^{z_{3}}
&\bullet \save[]+<0pt,-12pt>*{ }\restore \ar@{{}*{\cdot\;}{}}[r] 
&\bullet \save[]+<0pt,-12pt>*{ }\restore \ar@/^/[r]_{p_m}
&\bullet \save[]+<0pt,-12pt>*{e_{m,2}}\restore \ar@/^/[l] \ar@{-}[r]^{z_{m+1}}
& \ar@{{}*{\cdot\;}{}}[rr]
& 
&\bullet \save[]+<0pt,-12pt>*{ }\restore \ar@/^/[r]_{p_1}
&\bullet \save[]+<0pt,-12pt>*{e_{n-m+1,2}}\restore \ar@/^/[l] \ar@{-}[r]^{z_{n-m+2}}
&\bullet \save[]+<0pt,-12pt>*{ }\restore \ar@/^/[r]_{p_2}
&\bullet \save[]+<0pt,-12pt>*{e_{n-m+2,2}}\restore \ar@/^/[l]\ar@{-}[r]^{z_{n-m+3}}
&\bullet \save[]+<0pt,-12pt>*{ }\restore \ar@{{}*{\cdot\;}{}}[r] 
&\bullet \save[]+<0pt,-12pt>*{ }\restore \ar@/^/[r]_{p_m}
&\bullet \save[]+<0pt,-12pt>*{e_{n,2}}\restore \ar@/^/[l] \ar@(d,u)[lllllllllllllllddd]_{z_1=x}
&\\& \\& \\
&\ar@{{}*{\;\;\cdot}{}}[dd] \\
&\\
&\bullet \save[]+<0pt,-12pt>*{ }\restore \ar@/^/[r]_{p_1}
&\bullet \save[]+<0pt,-12pt>*{e_{1,d}}\restore \ar@/^/[l] \ar@{-}[r]^{z_{2}}
&\bullet \save[]+<0pt,-12pt>*{ }\restore \ar@/^/[r]_{p_2}
&\bullet \save[]+<0pt,-12pt>*{e_{2,d}}\restore \ar@/^/[l]\ar@{-}[r]^{z_{3}}
&\bullet \save[]+<0pt,-12pt>*{ }\restore \ar@{{}*{\cdot\;}{}}[r] 
&\bullet \save[]+<0pt,-12pt>*{ }\restore \ar@/^/[r]_{p_m}
&\bullet \save[]+<0pt,-12pt>*{e_{m,d}}\restore \ar@/^/[l] \ar@{-}[r]^{z_{m+1}}
& \ar@{{}*{\cdot\;}{}}[rr]
& 
&\bullet \save[]+<0pt,-12pt>*{ }\restore \ar@/^/[r]_{p_1}
&\bullet \save[]+<0pt,-12pt>*{e_{n-m+1,d}}\restore \ar@/^/[l] \ar@{-}[r]^{z_{n-m+2}}
&\bullet \save[]+<0pt,-12pt>*{ }\restore \ar@/^/[r]_{p_2}
&\bullet \save[]+<0pt,-12pt>*{e_{n-m+2,d}}\restore \ar@/^/[l]\ar@{-}[r]^{z_{n-m+3}}
&\bullet \save[]+<0pt,-12pt>*{ }\restore \ar@{{}*{\cdot\;}{}}[r] 
&\bullet \save[]+<0pt,-12pt>*{ }\restore \ar@/^/[r]_{p_m}
&\bullet \save[]+<+10pt,-11pt>*{e_{n,d}}\restore \ar@/^/[l]
 \ar `d/20pt[lllllllllllllllluuuu]
      `[lllllllllllllllluuuu]
      `[llllllllllllllluuuu]^{z_1=x}
       [llllllllllllllluuuu]
\save "2,2"."1,2"."9,2"*[F-]\frm{}
\restore
}\]
\end{figure}
\end{landscape}

\subsection{The classification theorem}
\begin{theorem}[\cite{DGO}, Theorem 5.7]\label{thm:classification}
$\;$
\begin{enumerate}
\item[(i)] The $A$-modules $V(\om), V(\om, f), V(\om,S,I_X), V(\om,i,w),$
and $V(\om,w,f)$ are indecomposable weight $A$-modules.
\item[(ii)] Every weight $A$-module $V$ such that $\dim_{\K_\mf}V_\mf<\infty$
 whenever $\mf$ belongs to a finite orbit, decomposes uniquely into a direct
 sum of modules isomorphic to those listed in (i).
\item[(iii)] The only isomorphisms between the listed modules are the following:
  \begin{itemize}
  \item If $f$ and $g$ are similar in $\K_\om[x,x^{-1};\tau]$, then
  \begin{equation}
  \label{eq:DGOiso1}
  V(\om,f)\simeq V(\om,g).
  \end{equation}
  \item If $f$ and $g$ are similar in $\K_\om[x;\tau^{n/m}]$, and $i\in\Z$, then
  \begin{equation}
  \label{eq:DGOiso2}
  V(\om,w,f)\simeq V(\om,w(mi),\tau^i(g)),
  \end{equation}
  where $m=|B_\om|$ and $n=|w|$.
  \end{itemize}
\end{enumerate}
\end{theorem}

\begin{remark}
In \cite{DGO}, $\tau^i$ is incorrectly missing from \eqref{eq:DGOiso2}.
In general, if $i$ is not a multiple of $n/m$, then $f$ is not similar to $\tau^i(f)$
in $\K_\om[x;\tau^{n/m}]$.
But for $g=f$, one can construct an isomorphism
$\varphi :  V(\om ,w(m) ,\tau (f)) \to V(\om ,w,f)$
determined by the conditions
\begin{align}
1)\quad &\varphi\big( [\si (\mf_0),e_{1,1}]\big) = [\si(\mf_0),e_{m+1,1}],\\
2)\quad &\varphi([\mf ,e_{k,s}]\big) \in
\begin{cases} 
 \oplus_{r=1}^d \K_\mf [\mf,e_{k+m,r}]& k+m\le n, \\
 \oplus_{r=1}^d \K_\mf [\mf,e_{k+m-n,r}]& k+m>n.
\end{cases}
\end{align}
\end{remark}
\begin{remark}
Taking $i=n/m$ in \eqref{eq:DGOiso2} we deduce that $f$ is similar to $\tau^{n/m}(f)$
in $P:=\K_\om[x;\tau^{n/m}]$. This isomorphism is explicitly given by
\begin{align*}
\varphi:P/P\tau^{n/m}(f)&\to P/Pf\\
g+P\tau^{n/m}(f)&\mapsto gx+Pf.
\end{align*}
This map is well defined since $\tau^{n/m}(f)x=xf$.
It is a homomorphism of left $P$-modules. Moreover, since
$f\neq x^d$ and is indecomposable, its constant term is nonzero.
Therefore $\varphi$ is surjective. Since dimensions agree, $\varphi$
is an isomorphism as claimed.
\end{remark}

The following description of the simple weight $A$-modules was also given.
\begin{theorem}[\cite{DGO}, Theorem 5.8] \label{thm:DGOthm2}
 The weight $A$-modules
 $V(\om), V(\om,f)$ for an irreducible $f\in\K_\om[x,x^{-1};\tau]$,
 $V(\om,S,\emptyset)$ for supportive interval $S\subseteq\om$ with $I(S)=\emptyset$,
 $V(\om,i,\ep)$ and $V(\om,w,f)$ for irreducible $f\in\K_\om[x;\tau^{n/m}]$
 and $w=x^m$ or $w=y^m$ where $m=|B_\om|$, are simple and each simple
 weight $A$-module is isomorphic to one from this list.
\end{theorem}

\section{Description of indecomposable weight modules having a non-degenerate admissible form}
\label{sec:description}
In this section we consider in turn each of the five types of indecomposable
modules from the DGO classification in Section \ref{sec:DGOclass} and
determine necessary and sufficient conditions, in terms of the parameters,
for the modules to be isomorphic to their finitistic dual
which, by Proposition \ref{prop:bij},
is equivalent to having a non-degenerate admissible form.
We will only consider the case when $\Supp(V)$ is contained in a
real orbit $\om$. The case of symmetric nonreal orbit will be left
for future studies.

The following lemma will be useful.
\begin{lemma}\label{lem:indec}
If $V$ is indecomposable, then so is $V^\sharp$.
\end{lemma}
\begin{proof}
We prove that if $V$ is decomposable, then so is $V^\sharp$. Then the result follows
since $V^{\sharp\sharp}\simeq V$, by Proposition \ref{eq:duality}.
Assume $V$ is decomposable and let $i_j:U_j\to V$, $j=1,2$,
be the inclusions of two submodules $U_j$ whose direct sum is $V$.
Let $W_j=\ker (i_j^\sharp)\subseteq V^\sharp$, $j=1,2$.
Let $\varphi\in W_1\cap W_2$. Then $i_1^\sharp(\varphi)=0=i_2^\sharp(\varphi)$.
Thus $\varphi(i_j(u))=0$ $\forall u\in U_j$, $j=1,2$. Since $V=i_1(U_1)+i_2(U_2)$
we deduce $\varphi=0$. Hence $W_1\cap W_2=0$.
Let $\varphi\in V^\sharp$ be arbitrary. Then $\varphi p_1+\varphi p_2=\varphi$,
where $p_j:V\to U_j$ are the projections. Also $i_1^\sharp(\varphi p_2)(v)=
(\varphi p_2)(i_1(v))=0 \forall v\in U_1$, and similarly $i_2^\sharp(\varphi p_1)=0$.
This proves that $V^\sharp=W_1+W_2$. 
\end{proof}

\subsection{Infinite orbit without breaks} \label{sec:inforbnobr}
\begin{theorem}\label{thm:inforbnobr}
Let $V=V(\om)$, where $\om$ is any infinite real orbit with $B_\om=\emptyset$.
Then $V^\sharp\simeq V$ and hence $V$ is pseudo-unitarizable.
\end{theorem}
\begin{proof}
We have $\Supp(V)=\om$. By the classification theorem, there is only one indecomposable module
whose support is contained in $\om$. By Lemma \ref{lem:indec}, $V^\sharp$ is indecomposable
and by Proposition \ref{prop:supp}, $\Supp(V^\sharp)=\Supp(V)=\om$.
Hence we conclude that $V^\sharp\simeq V$. By Proposition \ref{prop:bij}, $V$ is
pseudo-unitarizable.
\end{proof}
Let $\om$ be infinite real, $B_\om=\emptyset$, $V=V(\om)$.
We now determine all non-degenerate admissible forms on $V$,
and their index in the symmetric complex case.
Let $e_0\in V_{\mf(\om)}$, $e_0\neq 0$. Let $e_0^\sharp\in V^\sharp$ be defined
by $e_0^\sharp(e_0)=1_{\mf(\om)}$ and $e_0^\sharp(V_\mf)=0\;\forall \mf\in\om, \mf\neq\mf(\om)$.
Then $e_0^\sharp$ spans $(V^\sharp)_{\mf(\om)}$ over $\K_\om$ so
any isomorphism $\Phi:V\to V^\sharp$ must satisfy $\Phi(e_0)=\la e_0^\sharp$
for some nonzero $\la\in \K_\om$. Conversely, it is easy to see that for any nonzero
 $\la\in\K_\om$ there exists
a unique isomorphism $\Phi_\la: V\to V^\sharp$ satisfying $\Phi_\la(e_0)=\la e_0^\sharp$.
The set $\{e_n:=X^ne_0,\; e_{-n-1}:=Y^{n+1}e_0\; |\; n\in\Z_{\ge 0}\}$ is a basis for $V$ over $\K_\om$
and the
corresponding $\K_\om$-form $\Psi_\la$ (which is obtained using the bijections between
$\Hom_A(V,V^\sharp)$ and admissible forms in Proposition \ref{prop:bij}
and between admissible forms and $\K_\om$-forms in Proposition \ref{prop:fieldform2})
satisfies
\begin{align}
\Psi_\la(e_n,e_m)&=0,\quad m\neq n,\nonumber\\
\label{eq:Phila}
\Psi_\la(e_n,e_n)&=\begin{cases}
t\si^{-1}(t)\cdots\si^{-n+1}(t)\la, &n\ge 0,\\
\si(t)\si^2(t)\cdots\si^{-n}(t)\la,&n<0.
\end{cases}
\end{align}
To simplify notation we use here the natural $R$-module action on $\K_\om$. For example
 $t\la$ equals the product $(t+\mf(\om))\la$ in $\K_\om$.
From the formula \eqref{eq:Phila}, and the fact that $t^\ast=t$,
we see that the adjoint form $\Psi_\la^\sharp$
is equal to $\Psi_{\overline{\la}}$.

In the case when $\K_\om\simeq\C$ and conjugation is ordinary complex conjugation, we
associate to a symmetric form $\Psi_\la$, $\la\in\mathbb{R}$,
a scalar product on $V$ defined by
$(e_k,e_l)=\sgn\big(\Psi_\la(e_k,e_l)\big)\Psi_\la(e_k,e_l)$.
Then $\Psi_\la(v,w)=(Jv,w)\;\forall v,w\in V$, where $Je_k=\sgn\big(\Psi_\la(e_k,e_k)\big)e_k$.
$J$ is an involution operator in the sense that $J^2=\Id_V$ and that it is
self-adjoint with respect to the scalar product on $V$.
Therefore, (the completion of) $V$ together with $\Psi_\la$ is a Krein space (see [KS]).
Let $V_{\pm}=\{v\in V\; :\; Jv=\pm v\}$. Then $V=V_+\oplus V_-$.
We claim that any pair $(\dim V_+,\dim V_-)$ can occur. In fact,
consider
the sequence $(i_n)_{n\in\Z}$ where $i_n=\sgn\big(\Psi_\la(e_n,e_n)\big)$.
Then any sequence $(i_n)_{n\in\Z}\in\{1,-1\}^{\Z}$ can occur.
Indeed, let $R=\C[t_n \;|\; n\in\Z]$ be a polynomial
algebra in infinitely many indeterminates $t_n$. Let $t=t_0$, define $t_n^\ast=t_n$,
 $i^\ast=-i$ and extend $\ast$ to an $\mathbb{R}$-algebra automorphism of $R$. Let $\si(t_n)=t_{n+1}$
and let $\mf$ be the maximal ideal generated by $t_n-a_n, \;n\in\Z$, where $a_n\in\mathbb{R}$
are given by $a_n=i_{-n}i_{-n+1}, n\in\Z$.
Let $\om$ be the orbit containing $\mf$ and set
 $\mf(\om)=\mf$. The orbit $\om$ is infinite, real, and $B_\om=\emptyset$.
Then the sequence associated to the form $\Psi_{i_0}$ on $V(\om)$
equals $(i_n)_{n\in\Z}$.

\subsection{Infinite orbit with breaks}
\begin{theorem}\label{thm:inforbwbr}
Let $V=V(\om,S,I_X)$, where $\om\in\Om$ is infinite and real, $|B_\om|>0$,
$S\subseteq\om$ is a supportive interval, and $I_X\subseteq I(S)$.
Then $V^\sharp\simeq V(\om,S,I(S)\backslash I_X)$.
In particular $V$ is pseudo-unitarizable
iff $I(S)=\emptyset$ which is equivalent to $V$ being simple.
\end{theorem}
\begin{proof}
If $V^\sharp\simeq V$, then Proposition \ref{prop:semisimple}
and the fact that $V$ is indecomposable imply that $V$ must be simple.
The converse follows when we prove the more general statement that
$V^\sharp\simeq V(\om,S,I(S)\backslash I_X)$.

By Lemma \ref{lem:indec}, $V^\sharp$ is indecomposable
and by Proposition \ref{prop:supp} and that $\om$ is real,
$\Supp(V^\sharp)=\Supp(V)=S$.
So by the Drozd-Guzner-Ovsienko classification theorem, as stated
in Theorem \ref{thm:classification} in the present paper,
we deduce that $V^\sharp\simeq V(\om,S,J)$
for some subset $J$ of $I(S)$. It remains to prove
that, for $\mf\in I(S)$, $X(V^\sharp)_\mf\neq 0$ iff $XV_\mf=0$.

Suppose $\mf\in I(S)$ with $X(V^\sharp)_\mf=0$.
Let $\varphi\in (V^\sharp)_\mf$ be nonzero. Then, by Proposition \ref{prop:dualwtny},
 $\varphi |_{V_\nf}=0$
if $\nf\neq \mf$ and $\varphi(v)=1_\mf$ for some $v\in V_\mf$.
Let $u\in V_{\si(\mf)}$ be nonzero. We have
 $0=(X\varphi)(u)=\si\big(\varphi(Yu)\big)$. Thus $Yu=0$. Thus $u=Xv$ for some
nonzero $v\in V_\mf$, otherwise $V$ would be decomposable into
 $\big(\oplus_{n\le 0} V_{\si^n(\mf)}\big) \oplus \big(\oplus_{n>0} V_{\si^n(\mf)}\big)$.
This proves that $\mf\in I_X$, i.e. $XV_\mf\neq 0$. The converse is similar.

We conclude that indeed $V^\sharp\simeq V$ iff $I(S)=\emptyset$. Thus by Proposition
\ref{prop:bij}, $V$ is pseudo-unitarizable iff $I(S)=\emptyset$.
By Theorem \ref{thm:DGOthm2}, $V(\om,S,I_X)$ is simple iff $I(S)=\emptyset$.
\end{proof}

Let $\om\in\Om$ be real, infinite, $|B_\om|>0$. In this case $\om$ is torsion
trivial and thus there is a bijection between admissible forms and
admissible $\K_\om$-forms.
We now determine all possible non-degenerate admissible $\K_\om$-forms on
$V(\om,S,\emptyset)$ where $S$ is a supportive interval in $\om$
with $I(S)=\emptyset$.

 The subset $S\subseteq\om$ has either a maximal
or a minimal element (otherwise it would contain an inner break).
Assume $S$ has a maximal element $\nf_1$.
It is a break since $S$ is supportive. We can assume that $\mf(\om)=\nf_1$.
Let $e_0\in V_{\mf(\om)}$ be nonzero and $e_0^\sharp\in(V^\sharp)_{\mf(\om)}$
be such that $e_0^\sharp(e_0)=1_{\mf(\om)}$.
For $\la\in\K_\om$ there is a unique isomorphism
 $\Phi_\la:V\to V^\sharp$ given by $\Phi_\la(e_0)=\la e_0^\sharp$.
If $S$ has no minimal element,
 $V$ has a basis $\{e_{-n}:=Y^n e_0\; |\; n\ge 0\}$.
If $S$ has a minimal element $\nf_0$, then $\si^{-1}(\nf_0)\in B_\om$ and
 $V$ has a basis $\{e_{-n}:=Y^n e_0\; |\; 0\le n\le N-1\}$
where $\si^{-N}(\mf(\om))=\si^{-1}(\nf_0)$. The corresponding $\K_\om$-form $\Psi_\la$
calculated on the basis vectors gives
\begin{equation} \label{eq:Phila2}
\Phi_\la(e_{-n},e_{-m})=\si(t)\si^2(t)\cdots\si^n(t)\la\delta_{n,m}
\end{equation}
for $n,m\ge 0$.
If $S$ has no maximal element, but a minimal element $\nf_0$, then $\si^{-1}(\nf_0)\in B_\om$.
We choose $\mf(\om)=\nf_0$ in this case.
Then $V$ has a basis $\{e_n:=X^n e_0\;|\; n\ge 0\}$ and the corresponding
$\K_\om$-form $\Psi_\la$
satisfies
\begin{equation}\label{eq:Phila3}
\Psi_\la(e_n,e_m)=t\si^{-1}(t)\cdots\si^{-n+1}(t)\la\delta_{n,m}
\end{equation}
for $n,m\ge 0$.
We see that $\Psi_\la$ is symmetric iff $\overline{\la}=\la$.

\subsection{Finite orbit without breaks}
In this section we fix a finite orbit $\om\in\Om$ with $B_\om=\emptyset$.
In Theorem \ref{thm:Vomfdual}
we will describe the dual modules $V(\om,f)^\sharp$ for
indecomposable $f\in \K_\om[x,x^{-1};\tau]$.
First we make some preliminary observations.
Let $p=|\om|$ and put $P=\K_\om[x,x^{-1};\tau]$.
\begin{proposition}\label{prop:Vomfiso}
Let $B$ be the subalgebra of $A$ generated by $X^p, Y^p$ and all $r\in R$.
Let $I=B\mf(\om)B$ be the ideal in $B$ generated by $\mf(\om)$. Then
there is a ring isomorphism
 \[\psi : B/I \to P \]
given by
\[\psi(X^p + I) = \xi\cdot x,\quad
  \psi(Y^p + I) = x^{-1},\quad
  \psi(r + I) = r_{\mf(\om)} \quad\text{for $r\in R$}, \]
where
\begin{equation}
 \xi=\big(\si(t)\si^2(t)\cdots \si^p(t)\big)_{\mf(\om)}.
\end{equation}
\end{proposition}
\begin{proof}
It is straightforward to show that $B$ is isomorphic to 
the generalized Weyl algebra $R(\si^p, t')$ where $t'=t\si^{-1}(t)\cdots \si^{-p+1}(t)$ and that there
is a ring homomorphism $B\to P$ given by $X^p\mapsto \xi x$, $Y^p\mapsto x^{-1}$, 
and $r\mapsto r_{\mf(\om)}$ for all $r\in R$. Trivially $\mf(\om)$, hence
$I$, is contained in the kernel of that map.
The induced map $B/I\to P$ is the map $\psi$.
Assume $b+I\in B/I$ is in the kernel of $\psi$. Since both rings
involved, and $\psi$, are $\Z$-graded in a natural way, we can assume
$b=rX^{pk}$ or $b=rY^{pk}$, $k\ge 0$. We immediately get $k=0$, $r\in\mf(\om)$.
So $\psi$ is injective. That $\psi$ is surjective is easy to see.
\end{proof}
Let $V=V(\om,f)$, where
$f=\al_0+\al_1 x+\cdots +\al_d x^d\in P$, ($\al_0\neq 0, \al_d\neq 0$),
is indecomposable.
Since $\om$ is an orbit of length $p$, we have
$BV_{\mf(\om)}\subseteq V_{\mf(\om)}$. Also $IV_{\mf(\om)}=0$.
Thus $V_{\mf(\om)}$ becomes a module
over $B/I$ and, via the isomorphism in Proposition \ref{prop:Vomfiso},
a $P$-module. The following proposition describes this $P$-module.
\begin{proposition}\label{prop:Vomfchar}
\[V_{\mf(\om)}\simeq P/Pf\]
as $P$-modules.
\end{proposition}
\begin{proof}
Let $e_i=(0,\ldots,\overset{i}{1},\ldots, 0)\in V_{\mf(\om)}=(\K_\om)^d$.
By \eqref{eq:finorbnobrX}, if $1\le i<d$,
\begin{align*}
X^p e_i &= X^{p-1}\si(F_f t_{\mf(\om)} e_i) = \si^p(t_{\mf(\om)})X^{p-1}\si(e_{i+1})= \\
&=\si^p(t_{\mf(\om)})\si^{p-1}(t_{\si(\mf(\om))})X^{p-2}\si^2(e_{i+1})=\cdots=\\
&=\xi\cdot e_{i+1}.
\end{align*}
Thus
\begin{equation}\label{eq:VomfXpke1}
 (\xi^{-1}X^p)^k e_1= e_{k+1} \quad\text{for $k=0,1,\ldots,d-1$.}
\end{equation}
Also we have, by \eqref{eq:finorbnobrX},
\begin{equation}\label{eq:VomfXpked}
 \xi^{-1} X^p e_d=\sum_{k=0}^{d-1}\tau(-\al_k / \al_d)e_{k+1}.
\end{equation}
Using \eqref{eq:VomfXpke1} and \eqref{eq:VomfXpked} we get
\begin{align}\nonumber
\tau(f).e_1 &= \sum_{k=0}^d \tau(\al_k) x^k . e_1 =
 \sum_{k=0}^d \tau(\al_k) (\xi^{-1}X^p)^k e_1 = \\
 &= \sum_{k=0}^{d-1}\tau(\al_k) e_{k+1}+
\tau(\al_d)\sum_{k=0}^{d-1} \tau(-\al_k /\al_d )e_{k+1}=0.
 \label{eq:Vomffe1}
\end{align}
From \eqref{eq:VomfXpke1} and that
$\{e_1,\ldots,e_d\}$ generates $V_{\mf(\om)}$ as an $R$-module,
we see that the vector $e_1$ generates $V_{\mf(\om)}$ as a $P$-module.
By \eqref{eq:Vomffe1}, we get an epimorphism of $P$-modules
\begin{align*}
\psi: P/P\tau(f) &\to V_{\mf(\om)}\\
      h+P\tau(f) &\mapsto h.e_1
\end{align*}
Since $\dim_{\K_\om} V_{\mf(\om)}=d=\dim_{\K_\om} P/P\tau(f)$,
we deduce that $\psi$ is an isomorphism. Since $f$ is similar
to $\tau(f)$, it follows that $V_{\mf(\om)}\simeq P/Pf$.
\end{proof}
Now we come to the main result in this section.
\begin{theorem} \label{thm:Vomfdual}
Let $V=V(\om,f)$, where $\om$ is a finite and real orbit with $B_\om=\emptyset$
and $f=\al_0+\al_1x+\cdots+a_dx^d\in P=\K_\om[x,x^{-1}; \tau]$, $\al_0\neq 0\neq \al_d$,
is indecomposable. Then
\[V(\om,f)^\sharp\simeq V(\om,f^\sharp)\]
with
\begin{equation}
f^\sharp = \sum_{k=0}^d \{k\}_\xi\cdot \tau^k (\overline{\al_{d-k}})\cdot x^k,
\end{equation}
where
\begin{equation}
 \{k\}_\xi = \xi\tau(\xi)\cdots\tau^{k-1}(\xi) \quad\text{for $k\ge 0$},
\end{equation}
and
\begin{equation}
 \xi = \big(\si(t)\si^2(t)\cdots\si^p(t)\big)_{\mf(\om)}.
\end{equation}
In particular, $V$ is pseudo-unitarizable iff $f$ is similar to $f^\sharp$ in $P$.
\end{theorem}
\begin{proof}
By Lemma \ref{lem:indec} and Proposition \ref{prop:supp},
$V^\sharp$ is indecomposable and the support $\Supp(V^\sharp)=\om$. So
by Theorem \ref{thm:classification}, we know that
$V^\sharp\simeq V(\om,h)$ for some $h\in P$.
Then by Proposition \ref{prop:Vomfchar}, $(V^\sharp)_{\mf(\om)}\simeq P/Ph$.
Thus, it is enough to prove that
$(V^\sharp)_{\mf(\om)}\simeq P/Pf^\sharp$ as $P$-modules,
because then $h$ is similar to $f^\sharp$ which implies that $V^\sharp\simeq V(\om,f^\sharp)$
by the isomorphism \eqref{eq:DGOiso1}. Moreover, then it follows by Proposition \ref{prop:bij} and the isomorphism \eqref{eq:DGOiso1} that  $V$ is pseudo-unitarizable iff $f$ is
similar to $f^\sharp$ in $P$.

For this,
let $e_i=(0,\ldots,\overset{i}{1},\ldots, 0)\in V_{\mf(\om)}=(\K_\om)^d$,
and define $e_i^\sharp\in V^\sharp$ by $e_i^\sharp(V_\nf)=0$ for $\nf\in\om$,
$\nf\neq\mf(\om)$ and $e_i^\sharp(e_k)=\delta_{ik}\cdot 1_{\mf(\om)}$
for $i,k=1,\ldots d$. Since $\om$ is real, $e_i^\sharp\in (V^\sharp)_{\mf(\om)}$.
By relation \eqref{eq:finorbnobrY},
\begin{equation}
Y^p e_k = \begin{cases}
e_{k-1},& k>1,\\
F_f^{-1} e_1,& k=1.
\end{cases}
\end{equation}
It is easy to check that
\begin{equation}
F_f^{-1}e_1=-\al_0^{-1} (\al_1 e_1+\al_2 e_2+\cdots \al_d e_d).
\end{equation}
Thus for any $i,k=1,\ldots,d$,
\begin{align}
\nonumber
(X^p e_i^\sharp)(e_k)=\tau\big(e_i^\sharp(Y^pe_k)\big)&=
\begin{cases}
\delta_{i,k-1}\cdot 1_{\mf(\om)},& k>1,\\
\tau(-\overline{\al_i}/\overline{\al_0})\cdot 1_{\mf(\om)},& k=1,
\end{cases}\\
\label{eq:VomfXpei}
&=\big(e_{i+1}^\sharp - \tau (\overline{\al_i}/\overline{\al_0})\cdot e_1^\sharp\big)(e_k)
\end{align}
with the convention that $e_i^\sharp=0$ for $i>d$.
Let also $\al_i=0$ for $i>d$. We claim that
\begin{equation}\label{eq:Vomfdualann}
\sum_{k=0}^n \tau^{k+1}\big(\overline{\al_{n-k}} / \overline{\al_0} \big)\cdot X^{pk}
 e_1^\sharp = e_{n+1}^\sharp ,
\quad \text{for all $n\ge 0$.}
\end{equation}
We prove this by induction on $n$. For $n=0$ it is trivial.
Assume that
\[\sum_{k=0}^{n-1} \tau^{k+1}\big(\overline{\al_{n-1-k}} / \overline{\al_0}\big)\cdot X^{pk}
e_1^\sharp = e_n^\sharp .\]
Apply $X^p$ to both sides to get
\[\sum_{k=0}^{n-1} \tau^{k+2}\big(\overline{\al_{n-1-k}} / \overline{\al_0}\big)\cdot X^{p(k+1)}
e_1^\sharp = X^p e_n^\sharp .\]
Use that, by \eqref{eq:VomfXpei}, $X^p e_n^\sharp=e_{n+1}^\sharp-\tau(\overline{\al_n}/\overline{\al_0})\cdot e_1^\sharp$
in the right hand side,
add $\tau(\overline{\al_n}/\overline{\al_0})\cdot e_1^\sharp$
to both sides, and replace $k$ by $k-1$ in the sum in the left hand side to obtain
\[\sum_{k=1}^{n} \tau^{k+1}\big(\overline{\al_{n-k}} / \overline{\al_0} \big)\cdot X^{pk}
e_1^\sharp + \tau\big(\overline{\al_n} / \overline{\al_0}\big)\cdot
  e_1^\sharp = e_{n+1}^\sharp.\]
This proves \eqref{eq:Vomfdualann}.
From \eqref{eq:Vomfdualann} we see that
$e_1^\sharp$ generates $(V^\sharp)_{\mf(\om)}$ as a $P$-module
and that $g.e_1^\sharp=0$, where
\[g=\sum_{k=0}^d \tau^{k+1}(\overline{\al_{d-k}}/\overline{\al_0}) (\xi x)^k =
\sum_{k=0}^d \xi\tau(\xi)\cdots\tau^{k-1}(\xi)\cdot\tau^{k+1}(\overline{\al_{d-k}}/\overline{\al_0}) x^k
 \in P.\]
Thus, as in the proof of Proposition \ref{prop:Vomfchar}, 
 $(V^\sharp)_{\mf(\om)}\simeq P/Pg$ as $P$-modules.
Moreover, one verifies that
$\tau^{-1}(\xi)\cdot\tau^{-1}(g)\cdot\tau^{-1}(\xi)\overline{\al_0}=f^\sharp$.
Thus $g$ is similar to $f^\sharp$ and we conclude
that $(V^\sharp)_{\mf(\om)}\simeq P/Pf^\sharp$.
This finishes the proof of the theorem.
\end{proof}

\begin{remark}
The example in Section \ref{sec:ex_uqsl2}, concerning $U_q(\mathfrak{sl}_2)$,
shows that there exist non-simple indecomposable weight
modules $V(\om,f)$ which are pseudo-unitarizable.
This is in contrast to the case of bounded $\ast$-representations
of $\ast$-algebras on Hilbert spaces, that is, unitarizable modules
with respect to a positive definite form, where any unitarizable module
is semisimple. That example also shows that not all simple weight modules
are pseudo-unitarizable.
\end{remark}

\subsection{Finite orbit with breaks, first kind}
Recall that we defined
an automorphism of order two
of the monoid $\mathbf{D}$ by $x^\shrp =y$ and
 $y^\shrp =x$. For example, $(xxy)^\shrp=yyx$.
\begin{theorem}\label{thm:firstkind}
Let $\om$ be a finite real orbit with $m:=|B_\om|>0$, let $j\in \Z_m$
and let $w\in\mathbf{D}$. Then
 $V(\om,j,w)^\sharp\simeq V(\om,j,w^\shrp )$. In particular $V(\om,j,w)$
is pseudo-unitarizable iff $w=\ep$, the empty word (of length $n=0$),
which is equivalent to that $V(\om,j,w)$ is simple.
\end{theorem}
\begin{proof}
Define $\Phi : V(\om,j,w)\to V(\om,j,w^\shrp )^\sharp$ by
 $\Phi\big( [\mf,e_k] \big) = c_{\mf,k} [\mf,e_k^\sharp]$
where $[\mf,e_k^\sharp]\in V(\om,j,w^\shrp)^\sharp$ are defined
by $[\mf,e_k^\sharp] \big( [\nf,e_l] \big) = \delta_{\nf,\mf}\delta_{k,l}\cdot 1_\mf$
(where $1_\mf=1+\mf\in R/\mf\subseteq R_\om$)
and the coefficients $c_{\mf,k}\in R/\mf$ are nonzero, to be determined later.
Extend $\Phi$ to an $R$-module isomorphism.

Let $[\mf,e_k]$ be a basis vector of $V(\om,j,w)$.
Thus $j+k\equiv j(\mf) \;(\text{mod }m)$.
Write $w=z_1\cdots z_n$. Consider a basis vector of the form
 $[\si(\mf),e_l]\in V(\om,j,w^\shrp)$. We have
\begin{alignat*}{2}
\Big(&X\Phi\big([\mf,e_k]\big)\Big)\big([\si(\mf),e_l]\big)=   \\
&=\si\Big(c_{\mf,k}[\mf,e_k^\sharp ]\big( Y[\si(\mf),e_l]\big) \Big)
    \quad\qquad\qquad\text{by $A$-module str. on $V(\om,j,w^\shrp)^\sharp$}\\
&=
\begin{cases}
\si\Big( c_{\mf,k}[\mf,e_k^\sharp] \big([\mf,e_l] \big) \Big) ,&\mf\notin B_\om,\\
\si\Big( c_{\mf,k}[\mf,e_k^\sharp] \big([\mf,e_{l-1}] \big)\Big) ,&\mf\in B_\om\text{ and }z_l^\shrp=y,\\
0,&\text{otherwise}
\end{cases}
    \qquad\qquad\qquad\text{by \eqref{eq:FinOrbBr1Y}}\\
&=\begin{cases}
\si(c_{\mf,k})\delta_{kl}\cdot 1_{\si(\mf)},&\mf\notin B_\om,\\
\si(c_{\mf,k})\delta_{k,l-1}\cdot 1_{\si(\mf)},&\mf\in B_\om\text{ and }z_l=x,\\
0,&\text{otherwise}
\end{cases}
    \text{}\\
&=\begin{cases}
\si(c_{\mf,k})c_{\si(\mf),k}^{-1}
 \Big( \Phi\big([\si(\mf),e_k]\big) \Big) \big([\si(\mf),e_l]\big),&\mf\notin B_\om,\\
\si(c_{\mf,k})c_{\si(\mf),k+1}^{-1}
 \Big( \Phi\big([\si(\mf),e_{k+1}]\big) \Big) \big([\si(\mf),e_l]\big),&\mf\in B_\om\text{ and }z_{k+1}=x,\\
0,&\text{otherwise.}
\end{cases}
    \text{}\\
&=\Big(\Phi \big(X[\mf,e_k]\big)\Big)\big([\si(\mf),e_l]\big)
    \qquad\qquad\qquad\qquad
    \qquad\qquad\qquad\qquad\quad
    \text{by \eqref{eq:FinOrbBr1X}}
\end{alignat*}
if $c_{\mf,k}$ are chosen in such a way that
 $\si(c_{\mf,k})/c_{\si(\mf),k}=\si(t_\mf)$ when $\mf\notin B_\om$ 
and $\si(c_{\mf,k})/c_{\si(\mf),k+1}=1$ when $\mf\in B_\om$ and $z_{k+1}=x$.
On other basis vectors $[\nf,e_l]$, $\nf\neq\si(\mf)$, both sides are zero:
\[\Big(X\Phi\big([\mf,e_k]\big)\Big)\big([\nf,e_l ]\big) = 0 =
\Big(\Phi\big(X[\mf,e_k]\big)\Big)\big([\nf,e_l ]\big).
\]
With this choice of coefficients,
$\Phi$ commutes with the action of $X$.
For the action of $Y$, suppose $v$ is a basis vector of $V(\om,j,w)$
which is equal to $Xu$ for some $u$. Then
\[\Phi(Yv)=\Phi(YXu)=\Phi(tu)=t\Phi(u)=YX\Phi(u)=Y\Phi(Xu)=Y\Phi(v).\]
It remains to compare the results of applying $\Phi Y$ and $Y\Phi$ on basis vectors
which are not in the image of $X$. They have the form $[\si(\mf),e_k]$
where $\mf\in B_\om$ and $z_k\neq x$, i.e. $z_k=y$ or $k=0$.
Similarly to the previous calculation we get
\begin{align*}
\Big(Y\Phi\big ( [\si(\mf) & ,e_k]\big)\Big)\big([\mf,e_l]\big)=
 \si^{-1}\Big(c_{\si(\mf),k}[\si(\mf),e_k^\sharp]\big(X[\mf,e_l]\big)\Big)=\\
&=\begin{cases}
\si^{-1}\Big(c_{\si(\mf),k}[\si(\mf),e_k^\sharp] \big( [\si(\mf),e_{l+1}]\big)\Big),& z_{l+1}^\shrp =x,\\
0,&\text{otherwise}
\end{cases}\\
&=\begin{cases}
\si^{-1}(c_{\si(\mf),k})\delta_{k,l+1}\cdot 1_\mf,& z_{l+1}=y,\\
0,&\text{otherwise}
\end{cases}\\
&=
 \begin{cases} \si^{-1}(c_{\si(\mf),k})c_{\mf,k-1}^{-1}
 \Big( \Phi\big( [\mf,e_{k-1}] \big) \Big) \big([\mf,e_l]\big),& z_k=y,\\
0,&\text{otherwise}
\end{cases}\\
&=\Big( \Phi\big(Y[\si(\mf),e_k]\big) \Big) \big([\mf,e_l]\big)
\end{align*}
if the coefficients are chosen such that $\si^{-1}(c_{\si(\mf),k})/c_{\mf,k-1}=1$ when
 $\mf\in B_\om$ and $z_k=y$. 
Choosing the coefficients in this way, which is always possible,
 $\Phi$ becomes an isomorphism of $A$-modules.
\end{proof}

\begin{example}
Assume that $\om\in\Om$ is real and $p=|\om|=7$. Pick $\nf\in\om$. Then
 $\om=\{\si^j(\nf)\; |\; j=0,\ldots,6\}$. Suppose that $B_\om=\{\mf_0:=\si^2(\nf),
\mf_1:=\si^4(\nf), \mf_2:=\si^6(\nf)\}$, so $m=|B_\om|=3$.

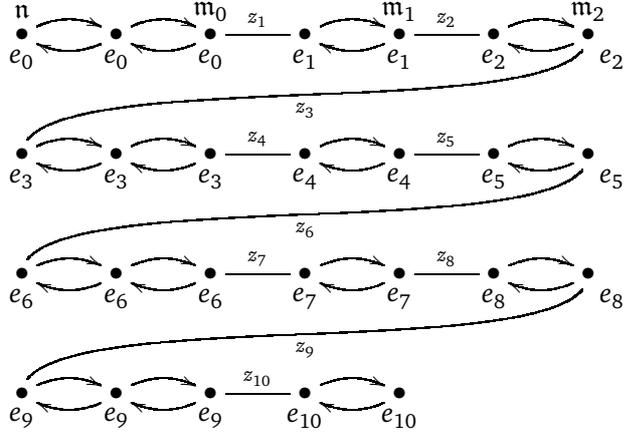
\begin{figure}
\caption{Weight diagram for $V(\om,j,w)$ where $j=0$ and $w=z_1z_2\cdots z_{10}$}
\[\xymatrix@R-10pt{
&\bullet \save[]+<0pt,-9pt>*{e_0}+<0pt,18pt>*{\nf}\restore \ar@/^/[r]
&\bullet \save[]+<0pt,-9pt>*{e_0}+<0pt,18pt>*{ }\restore \ar@/^/[l] \ar@/^/[r]
&\bullet \save[]+<0pt,-9pt>*{e_0}+<0pt,18pt>*{\mf_0}\restore \ar@/^/[l] \ar@{-}[r]^{z_1}
&\bullet \save[]+<0pt,-9pt>*{e_1}+<0pt,18pt>*{ }\restore \ar@/^/[r]
&\bullet \save[]+<0pt,-9pt>*{e_1}+<0pt,18pt>*{\mf_1 }\restore \ar@/^/[l] \ar@{-}[r]^{z_2}
&\bullet \save[]+<0pt,-9pt>*{e_2}+<0pt,18pt>*{ }\restore \ar@/^/[r]
&\bullet \save[]+<9pt,-9pt>*{e_2}+<-9pt,18pt>*{\mf_2 }\restore \ar@/^/[l]
 \ar@(d,u)@{-}[ddllllll]^{z_3} \\
 \\
&\bullet \save[]+<0pt,-9pt>*{e_3}+<0pt,18pt>*{ }\restore \ar@/^/[r]
&\bullet \save[]+<0pt,-9pt>*{e_3}+<0pt,18pt>*{ }\restore \ar@/^/[l] \ar@/^/[r]
&\bullet \save[]+<0pt,-9pt>*{e_3}+<0pt,18pt>*{}\restore \ar@/^/[l] \ar@{-}[r]^{z_4}
&\bullet \save[]+<0pt,-9pt>*{e_4}+<0pt,18pt>*{ }\restore \ar@/^/[r]
&\bullet \save[]+<0pt,-9pt>*{e_4}+<0pt,18pt>*{ }\restore \ar@/^/[l] \ar@{-}[r]^{z_5}
&\bullet \save[]+<0pt,-9pt>*{e_5}+<0pt,18pt>*{ }\restore \ar@/^/[r]
&\bullet \save[]+<9pt,-9pt>*{e_5}+<-9pt,18pt>*{ }\restore \ar@/^/[l]
 \ar@(d,u)@{-}[ddllllll]^{z_6} \\
\\
&\bullet \save[]+<0pt,-9pt>*{e_6}+<0pt,18pt>*{}\restore \ar@/^/[r]
&\bullet \save[]+<0pt,-9pt>*{e_6}+<0pt,18pt>*{ }\restore \ar@/^/[l] \ar@/^/[r]
&\bullet \save[]+<0pt,-9pt>*{e_6}+<0pt,18pt>*{}\restore \ar@/^/[l] \ar@{-}[r]^{z_7}
&\bullet \save[]+<0pt,-9pt>*{e_7}+<0pt,18pt>*{ }\restore \ar@/^/[r]
&\bullet \save[]+<0pt,-9pt>*{e_7}+<0pt,18pt>*{ }\restore \ar@/^/[l] \ar@{-}[r]^{z_8}
&\bullet \save[]+<0pt,-9pt>*{e_8}+<0pt,18pt>*{ }\restore \ar@/^/[r]
&\bullet \save[]+<9pt,-9pt>*{e_8}+<-9pt,18pt>*{ }\restore \ar@/^/[l]
 \ar@(d,u)@{-}[ddllllll]^{z_9} \\
\\
&\bullet \save[]+<0pt,-9pt>*{e_9}+<0pt,18pt>*{}\restore \ar@/^/[r]
&\bullet \save[]+<0pt,-9pt>*{e_9}+<0pt,18pt>*{ }\restore \ar@/^/[l] \ar@/^/[r]
&\bullet \save[]+<0pt,-9pt>*{e_9}+<0pt,18pt>*{}\restore \ar@/^/[l] \ar@{-}[r]^{z_{10}}
&\bullet \save[]+<0pt,-9pt>*{e_{10}}+<0pt,18pt>*{ }\restore \ar@/^/[r]
&\bullet \save[]+<0pt,-9pt>*{e_{10}}+<0pt,18pt>*{ }\restore \ar@/^/[l]
}\]
\end{figure}

With $\om$ as above, there are three modules of the form $V(\om,j,\ep)$ corresponding
to $j=0,1,2$. For example, $V(\om,1,\ep)$ is two-dimensional with
basis $\{[\si^{-1}(\mf_1),e_1], [\mf_1,e_1] \}$.

\end{example}

In general, let $j\in\Z_m$ and $V=V(\om,j,\ep)$.
We determine all non-degenerate admissible forms on $V$.
$V$ has a basis
\[\{v_k:=[\si^{-k}(\mf_j),e_j] \; |\; k=0,1,\ldots,p_j-1\},\]
where $p_j>0$ is minimal such that $\si^{p_j}(\mf_{j-1})=\mf_j$.
Any $A$-module isomorphism $V\to V^\sharp$ has the form
$\Phi_\la(v_0)=\la v_0^\sharp$
for some $\la\in\K_{\mf_j}$, where $v_0^\sharp=[\mf_j,e_j^\sharp]$.
The corresponding admissible form satisfies
\begin{multline}\label{eq:PhilaFinite}
\widehat{\Phi_\la}(v_n,v_m)=
\widehat{\Phi_\la}(Y^nv_0,Y^nv_0)\delta_{n,m}=
\si^{-n}\big(\widehat{\Phi_\la}(X^nY^nv_0,v_0)\big)\delta_{n,m}=\\=
\si^{-n}\big(\si(t)\si^2(t)\cdots\si^n(t)\la\big)\delta_{n,m}
\end{multline}
for $n,m=0,1,\ldots,p_j-1$. It is clearly non-degenerate iff $\la\neq 0$.

Suppose that $\om$ is torsion trivial (recall that, by Definition \ref{dfn:torsion_trivial},
this means that if $\sigma^n(\mf)=\mf$ for some $n\in\Z$ and some $\mf\in\om$,
then the induced map from $R/\mf$ to itself is the identity).
Choose $\mf(\om)=\mf_j$. Suppose that
 $\K_\om\simeq\C$ and that conjugation is usual complex conjugation
and assume that $\la\in \mathbb{R}$. Let $\Psi_\la$ be the associated
symmetric $\C$-form as described in Proposition \ref{prop:fieldform2}.
We have
\[
\Psi_\la(v_n,v_m)=\big(\si(t)\si^2(t)\cdots\si^n(t)\big)_{\mf(\om)}\la\delta_{n,m}
\]
for $n,m=0,1,\ldots,p_j-1$.
Let us calculate the index $(n_+,n_-)$, (i.e. $n_+$ $(n_-)$ is the number
of positive (negative) eigenvalues) of the form $\Psi_\la$.
Let $a_0=\la$ and $a_i=\si^i(t)+\mf(\om)\in\mathbb{R}$, $i=1,\ldots, p_j-1$.
Let $0\le s_1<s_2<\cdots<s_r\le p_j-1$ be the integers $i$
for which $a_i<0$ and put $s_i=0$ for $i\le 0$ and put $s_i=p_j$ for $i>r$.
Then one can check that $\Psi_\la$ has index
 $\big(\sum_{i\in\Z} (s_{2i+1}-s_{2i}), \sum_{i\in\Z} (s_{2i}-s_{2i-1})\big)$. 
For example, if $p_j=7$ and $\sgn(\la,a_1,a_2,\ldots,a_6)=(+,+,-,+,+,-,-)$,
then the index of $\Psi_\la$ is $(2+1,3+1)=(3,4)$.
All possible indices can occur. This can be seen as in Section \ref{sec:inforbnobr}.

\subsection{Finite orbit with breaks, second kind}
For $r\in R$ and $\mf\in\Max(R)$, we put $r_\mf=r+\mf\in R/\mf$ for brevity.
First we prove a theorem which partially describes the finitistic dual of a
module of the form $V(\om,w,f)$.
\begin{theorem}\label{thm:2ndkind_w}
Let $\om\in\Om$ be a finite real orbit.
Let $V=V(\om,w,f)$ where $w=z_1z_2\cdots z_n$ is an $m$-word,
and $f=a_1+a_2x+\cdots+a_dx^{d-1}+x^d\in \K_\om[x;\tau^{n/m}]$ is any element
with $a_1\neq 0$.
Then $V^\sharp\simeq V(\om,w^\shrp,g)$ for some $g\in\K_\om[x;\tau^{n/m}]$.
\end{theorem}
\begin{proof}
For simplicity, we will assume that $z_1=x$. The proof of the case $z_1=y$ is similar.

\textbf{Step 1.} We find the action of $X$ and $Y$ on a dual basis in $V^\sharp$.
Relations \eqref{eq:2ndkindX}-\eqref{eq:2ndkindY} for the module $V$ can be written
\begin{align}
\label{eq:2ndkindXpf}
X[\mf,e_{ks}]&=\begin{cases}
\si(t_{\mf}) \cdot [\si(\mf),e_{ks}],&\mf\notin B_\om,\\
[\si(\mf),e_{k+1,s}],&\mf\in B_\om, k< n, z_{k+1}=x,\\
0,&\mf\in B_\om, k< n, z_{k+1}=y,\\
[\si(\mf),e_{1,s+1}],&\mf\in B_\om, k=n, s< d,\\
-\sum_{i=1}^d \si(a_i)\cdot [\si(\mf),e_{1i}],&\mf\in B_\om, k=n, s=d,
\end{cases}\\
\label{eq:2ndkindYpf}
Y[\mf,e_{ks}]&=\begin{cases}
[\si^{-1}(\mf),e_{ks}],&  \si^{-1}(\mf)\notin B_\om,\\
[\si^{-1}(\mf),e_{k-1,s}],& \si^{-1}(\mf)\in B_\om, k>1, z_k=y,\\
0,&  \si^{-1}(\mf)\in B_\om, k>1, z_k=x,\\
0,&  \si^{-1}(\mf)\in B_\om, k=1.
\end{cases}
\end{align}
Let
\[\big\{[\mf,e_{ks}^\sharp]\; |\; s=1,\ldots,d,\; k=1,\ldots,n,\; k\equiv j(\mf)\;(\text{mod }m)\big\}\]
be the dual basis in $V^\sharp$, defined by requiring (recall that $1_\mf$ denotes $1+\mf\in R/\mf$)
\begin{equation}\label{eq:2ndkind_dualbasis}
 [\mf,e_{ks}^\sharp] \big([\nf,e_{lr}]\big)=
\begin{cases}
1_\mf,&\text{if $\mf=\nf, k=l, s=r$,}\\
0,&\text{otherwise,}
\end{cases}
\end{equation}
and $[\mf,e_{ks}^\sharp]$ to be additive and
 $[\mf,e_{ks}^\sharp](rv)=r^*\cdot[\mf,e_{ks}^\sharp](v)$
for any $r\in R$, $v\in V$.
Then the following relations hold for the action of $X$ and $Y$ on this dual basis:
\begin{equation}
\label{eq:2ndkind_dualX}
X [\mf,e_{ks}^\sharp] =\begin{cases}
[\si(\mf),e_{ks}^\sharp],&\mf\notin B_\om,\\
[\si(\mf),e_{k+1,s}^\sharp],&\mf\in B_\om, k<n, z_{k+1}=y,\\
0,&\text{otherwise,}
\end{cases}
\end{equation}
\begin{multline} \label{eq:2ndkind_dualY}
Y[\mf,e_{ks}^\sharp] = \\
 =\begin{cases}
t_{\si^{-1}(\mf)}\cdot [\si^{-1}(\mf),e_{ks}^\sharp],&\si^{-1}(\mf)\notin B_\om,\\
[\si^{-1}(\mf),e_{k-1,s}^\sharp],&\si^{-1}(\mf)\in B_\om, k>1, z_k=x,\\
0,&\si^{-1}(\mf)\in B_\om, k>1, z_k=y,\\
[\si^{-1}(\mf),e_{n,s-1}^\sharp]-\overline{a_s} \cdot [\si^{-1}(\mf),e_{nd}^\sharp],
 &\si^{-1}(\mf)\in B_\om, k=1, s>1,\\
-\overline{a_1} \cdot [\si^{-1}(\mf),e_{nd}^\sharp],& \si^{-1}(\mf)\in B_\om, k=1, s=1.
\end{cases}
\end{multline}
Let us prove the first case in \eqref{eq:2ndkind_dualY}. If $\si^{-1}(\mf)\notin B_\om$, then
\begin{align*}
\big(Y[\mf,e_{ks}^\sharp]\big)&\big([\si^{-1}(\mf),e_{lr}]\big)=\\
&=\si^{-1}\Big([\mf,e_{ks}^\sharp]\big(X[\si^{-1}(\mf),e_{lr}]\big)\Big)
 &\text{by $A$-module str. of $V^\sharp$,}\\
&=\si^{-1}\Big([\mf,e_{ks}^\sharp]\big(\si(t)\cdot [\mf,e_{lr}]\big)\Big)
 &\text{by \eqref{eq:2ndkindXpf},}\\
&=\si^{-1}\Big(\si(t)^*\cdot [\mf,e_{ks}^\sharp]\big([\mf,e_{lr}]\big)\Big)
 &\text{by $R$-antilinearity,}\\
&=t\cdot \delta_{kl}\delta_{sr}\cdot \si^{-1}(1_{\mf})
 &\text{by \eqref{eq:2ndkind_dualbasis},}\\
&=t\cdot [\si^{-1}(\mf),e_{ks}^\sharp]\big([\si^{-1}(\mf),e_{lr}]\big)
 &\text{by \eqref{eq:2ndkind_dualbasis}.}
\end{align*}
Furthermore, if $\nf\neq\si^{-1}(\mf)$ then
\[\big(Y[\mf,e_{ks}^\sharp]\big)\big([\nf,e_{lr}]\big)
=\si^{-1}\Big([\mf,e_{ks}^\sharp]\big(X[\nf,e_{lr}]\big)\Big)=0
=t\cdot [\si^{-1}(\mf),e_{ks}^\sharp]\big([\nf,e_{lr}]\big)\]
using that $X[\nf,e_{lr}]\in V_{\si(\nf)}$ and \eqref{eq:2ndkind_dualbasis}.
This proves that
 $Y[\mf,e_{ks}^\sharp]=t\cdot [\si^{-1}(\mf),e_{ks}^\sharp]=
t_{\si^{-1}(\mf)}\cdot [\si^{-1}(\mf),e_{ks}^\sharp]$ if $\si^{-1}(\mf)\notin B_\om$.

For the last two cases in \eqref{eq:2ndkind_dualY}, let us first note that
if $\si^{-1}(\mf)\in B_\om$ and $j(\si^{-1}(\mf))\equiv n\equiv 0\; (\text{mod $m$})$
then in fact $\si^{-1}(\mf)=\mf_0$. We have
\begin{align*}
&\big(Y[\si(\mf_0),e_{1s}^\sharp]\big)\big([\mf_0,e_{lr}]\big)=\\
&=\si^{-1}\Big([\si(\mf_0),e_{1s}^\sharp]\big(X[\mf_0,e_{lr}]\big)\Big)
 \qquad\text{by $A$-module str. of $V^\sharp$,}\\
&=\si^{-1}\Big([\si(\mf_0),e_{1s}^\sharp]\big(
 [\si(\mf_0),e_{1,r+1}]\delta_{ln}\delta_{r<d}
  - \si(a_s)[\si(\mf_0),e_{1s}]\delta_{ln}\delta_{rd}\big)\Big)\\
&=\delta_{s-1,r}\delta_{s>1}\delta_{ln}1_{\mf_0}-\overline{a_s}\delta_{ln}\delta_{rd} 1_{\mf_0}\\
&=\big([\mf_0,e_{n,s-1}^\sharp]\delta_{s>1}-\overline{a_s}\cdot [\mf_0,e_{nd}^\sharp]\big)
 \big([\mf_0,e_{lr}]\big).
\end{align*}
The other cases in \eqref{eq:2ndkind_dualX},\eqref{eq:2ndkind_dualY} are easily checked.

\textbf{Step 2.}
We construct a basis $[\mf,f_{ks}]$ for $V^{\sharp}$ such that
 $[\mf,e_{ks}]\mapsto [\mf,f_{ks}]$ is an isomorphism from 
 $V(\om,w^\shrp,g)$ to $V^{\sharp}$ for some $g$.
We have a decomposition
\begin{equation}\label{eq:Vsharpdecomp}
(V^\sharp)_\mf=\bigoplus_{\substack{1\le k\le n,\\ k\equiv j(\mf)\;\text{(mod $m$)}}} (V^\sharp)_\mf^{(k)}
\qquad\text{for any $\mf\in\om$,}
\end{equation}
\begin{equation}
(V^\sharp)_\mf^{(k)}=\oplus_{s=1}^d \K_\mf [\mf, e_{ks}^\sharp].
\end{equation}
Note that, if $k>1$ and $z_k^\shrp=y$ then
 $Y:(V^\sharp)_\mf^{(k)}\to (V^\sharp)_{\si^{-1}(\mf)}^{(k-1)}$
 is bijective, where $\si^{-1}(\mf)\in B_\om$ is the unique break such that $j(\mf)\equiv k\;(\text{mod }m)$.
Indeed this is trivial since $Y[\mf,e_{ks}^\sharp]=[\si^{-1}(\mf),e_{k-1,s}^\sharp]$
for $s=1,\ldots,d$ by the second case in \eqref{eq:2ndkind_dualY}.
Also, $Y:(V^\sharp)_{\si(\mf_0)}^{(1)}\to (V^\sharp)_{\mf_0}^{(n)}$
is bijective by the fourth and fifth case in \eqref{eq:2ndkind_dualY},
using the assumption that $a_1\neq 0$.

Put
\begin{equation}\label{eq:2ndkind_fdef1_w}
 [\si(\mf_0),f_{11}] = [\si(\mf_0),e_{11}^\sharp]
\end{equation}
and recursively
\begin{equation}\label{eq:2ndkind_fdef2_w}
 [\mf,f_{ks}] =\begin{cases}
\si(t)_\mf^{-1} X[\si^{-1}(\mf),f_{ks}],
  & \si^{-1}(\mf)\notin B_\om,\\
X[\si^{-1}(\mf),f_{k-1,s}],
  &\si^{-1}(\mf)\in B_\om, z_k^\shrp=x (\Rightarrow k>1),\\
\big( Y|_{(V^\sharp)_\mf^{(k)}}\big)^{-1} [\si^{-1}(\mf),f_{k-1,s}],
  &\si^{-1}(\mf)\in B_\om, k>1, z_k^\shrp=y,\\
\big( Y\big|_{(V^\sharp)_\mf^{(1)}}\big)^{-1}[\si^{-1}(\mf),f_{n,s-1}],
  &\si^{-1}(\mf)\in B_\om, k=1.
\end{cases}
\end{equation}
Induction shows that each $[\mf,f_{ks}]$ is a linear combination of $[\mf,e_{kr}^\sharp]$ where
 $1\le r\le s$ and the coefficient of $[\mf,e_{ks}^\sharp]$ is nonzero.
 Thus $\big\{[\mf,f_{ks}]\big\}_{s=1}^{d}$ is a basis for $(V^\sharp)_{\mf}^{(k)}$.

We prove that there exists a $g\in\K_\om[x;\tau^{n/m}]$ such 
that the $R$-module isomorphism $\ph:V(\om,w^\shrp,g)\to V^\sharp$ defined
by $\ph([\mf,e_{ks}])=[\mf,f_{ks}]$ is an $A$-module isomorphism.
By \eqref{eq:2ndkindX},
\begin{align}
\nonumber
\ph(X[\mf,e_{ks}])&=\begin{cases}
\ph\big(\si(t)_{\si(\mf)}\cdot[\si(\mf),e_{ks}]\big),&\mf\notin B_\om,\\
\ph\big([\si(\mf),e_{k+1,s}]\big),&\mf\in B_\om, k<n, z_{k+1}^\shrp=x,\\
0,&\text{otherwise (since $z_1^\shrp=y$),}
\end{cases}\\
\label{eq:kalkyl762}
&=\begin{cases}
\si(t)_{\si(\mf)}\cdot [\si(\mf),f_{ks}],&\mf\notin B_\om,\\
[\si(\mf),f_{k+1,s}],&\mf\in B_\om, k<n, z_{k+1}^\shrp=x,\\
0,&\text{otherwise},
\end{cases}
\end{align}
while
 $X\ph\big([\mf,e_{ks}]\big)=X[\mf,f_{ks}]$.
By the recursive definition of $[\mf,f_{ks}]$,
the vector $X[\mf,f_{ks}]$ equals the right hand side of
\eqref{eq:kalkyl762}.
For example, $[\si(\mf),f_{ks}]=\si(t)_{\si(\mf)}^{-1}\cdot X[\mf,f_{ks}]$
if $\mf\notin B_\om$ by the first case in \eqref{eq:2ndkind_fdef2_w},
which gives $X[\mf,f_{ks}]=\si(t)_{\si(\mf)}\cdot [\si(\mf),f_{ks}]$.
 Similarly, by \eqref{eq:2ndkindY} and the construction of
the basis $[\mf,f_{ks}]$, $\ph\big(Y[\mf,e_{ks}]\big)=Y\ph\big([\mf,e_{ks}]\big)$
when $k>1$ or $s>1$ or $\mf\neq\si(\mf_0)$. For the last case, $k=s=1$ and $\mf=\si(\mf_0)$,
we know that
 $Y:(V^\sharp)_{\si(\mf_0)}^{(1)}\to (V^\sharp)_{\mf_0}^{(n)}$
is bijective. Thus, since $\big\{[\mf_0,f_{ns}]\big\}_{s=1}^d$
is a basis for $(V^\sharp)_{\mf_0}^{(n)}$,
\[Y\ph\big([\si(\mf_0),e_{11}]\big)=Y[\si(\mf_0),f_{11}]
=-\sum_{r=1}^d c_r^\circ\cdot [\mf_0,f_{nr}]\]
for some constants $c_r^\circ\in\K_\om$. Put $c_{d+1-r}=\tau^{r-d}(c_r^\circ)$ for
$r=1,\ldots,d$ and
choose $g=c_1+c_2x+\cdots+c_d x^{d-1}+x_d$.
Since $z_1^\shrp=y$, relation \eqref{eq:2ndkindY} gives that,
in $V(\om,w^\shrp,g)$ we have $Y[\si(\mf_0),e_{11}]=-\sum_{r=1}^d c_r^\circ [\mf_0,e_{nr}]$
and thus
\[\ph\big(Y[\si(\mf_0),e_{11}]\big)=\ph\big(-\sum_{r=1}^d c_r^\circ [\mf_0,e_{nr}]\big)=
-\sum_{r=1}^d c_r^\circ [\mf_0,f_{nr}].\]
This finishes the proof that $V^\sharp\simeq V(\om,w^\shrp,g)$ for
some $g$.
\end{proof}

As a corollary we get a necessary condition on the word $w$
for a module $V(\om,w,f)$ to be isomorphic to its finitistic dual.

\begin{corollary}\label{cor:2ndkind}
Let $\om$ be a finite real orbit. Let $V=V(\om,w,f)$ where $w=z_1z_2\cdots z_n$
is a non-periodic $m$-word, and $f=a_1+a_2x+\cdots +a_d x^{d-1}+x^d\neq x^d$
is indecomposable in $\K_\om[x;\tau^{n/m}]$.
If $V\simeq V^\sharp$ then $w=w_0w_0^\shrp$, where $w_0$ is an $m$-word.
\end{corollary}
\begin{proof}
Since $f$ is indecomposable and $f\neq x^d$ we have $a_1\neq 0$.
If $V\simeq V^\sharp$ then by Theorem \ref{thm:2ndkind_w}, $V\simeq V(\om,w^\shrp,g)$
for some $g\in\K_\om[x;\tau^{n/m}]$. Thus by the classification
in Theorem \ref{thm:classification} we must have
 $w(lm)=w^\shrp$
for some integer $l\ge 0$, chosen minimal. Clearly, $lm<n$.
 Since the operation $\shrp$ on the monoid $\mathbf{D}$ commutes with the $\Z$-action,
 we have
\begin{equation}\label{eq:w0cor}
 w(lm+k)=w(k)^\shrp\qquad\forall\; k\in\Z.
\end{equation}
 We claim that $2lm\le n$. Otherwise $lm<n<2lm$ and thus
 $0<n-lm<lm$. Also, $w(n-lm)=w(-lm)=w^\shrp$ since $w=w(-lm+lm)=w(-lm)^\shrp$
 by \eqref{eq:w0cor} with $k=-lm$. Thus the properties of
 the number $\frac{n}{m}-l$ contradicts the minimality of $l$.
Therefore $2lm\le n$ as claimed.

Now let $k=\mathrm{GCD}(2lm,n)$.
Trivially $w(n)=w$, and by \eqref{eq:w0cor},
 $w(2lm)=w(lm)^\shrp=w$. Hence $w(k)=w$ also. But $k | n$ and thus
 $w=(z_1z_2\cdots z_k)^{n/k}$. However $w$ is non-periodic and thus
 $n=k$, forcing $n=2lm$ so $w=w_0w_0^\shrp$ where $w_0=z_1z_2\cdots z_{lm}$ is
 an $m$-word.
\end{proof}

The following result describes the finitistic dual of a module
$V(\om,w,f)$ when $w$ has the special form $w=w_0w_0^\shrp$ as in
 Corollary \ref{cor:2ndkind}.

\begin{theorem}\label{thm:2ndkind_f}
Let $\om\in\Om$ be a finite real orbit with $m:=|B_\om|>0$.
Let $w_0\in\mathbf{D}\backslash\{\ep\}$ be an $m$-word and
put $l=|w_0|/m$ and $n=2|w_0|$.
Let $V=V(\om,w_0w_0^\shrp, f)$ where 
$f=\al_0+\al_1x+\cdots+\al_{d-1}x^{d-1}+\al_d x^d\in \K_\om[x;\tau^{n/m}]$ is any
element with $\al_0\neq 0\neq \al_d$.
Then $V^\sharp\simeq V(\om,w_0w_0^\shrp,f^\sharp)$, where
\begin{equation}\label{eq:2nd_fsharp}
f^\sharp =
\sum_{k=0}^d \{2lk\}\cdot \tau^{(2k+1)l}
 \big(\overline{\al_{d-k}} \big)
\cdot x^k.
\end{equation}
Here $\{k\}$ is a Pochhammer-type symbol:
\begin{equation}
\{k\}=\{k\}_{q,\tau}=q\tau(q)\cdots\tau^{k-1}(q)\in\K_\om,\quad k\in\Z_{\ge 0},
\end{equation}
where $q\in\K_\om\backslash\{0\}$ is given by
\begin{equation}\label{eq:qdef}
q=\si^{p_2+p_3+\cdots+p_m}(t_1)\si^{p_3+p_4+\cdots+p_m}(t_2)
\cdots \si^{p_m}(t_{m-1})t_m,
\end{equation}
\begin{equation}\label{eq:tidef}
t_i=\big(\si(t)\si^2(t)\cdots\si^{p_i-1}(t)\big)_{\mf_i}\qquad
\text{for $i=1,\ldots,m$,}
\end{equation}
where $p_i\in\Z_{>0}$ are minimal such that $\si^{p_i}(\mf_{i-1})=\mf_i$, $i=1,\ldots,m$.
\end{theorem}

Combining Corollary \ref{cor:2ndkind}, Theorem \ref{thm:2ndkind_f}
and Proposition \ref{prop:bij},
we obtain the following, which is the main result in this section.
\begin{theorem}\label{thm:2ndkind}
Let $V$ be any indecomposable weight $A$-module of the type $V(\om,w,f)$ with $\om$ real.
Thus $\om\in\Om$ is a finite real orbit with $m:=|B_\om|>0$,
$w\in\mathbf{D}\backslash\{\ep \}$ is a non-periodic $m$-word,
and $f=\al_0+\al_1 x+\cdots+\al_d x^d \in \K_\om[x;\tau^{n/m}]$,
$\al_d\neq 0$, is an indecomposable element not equal to $x^d$.
Then $V$ is pseudo-unitarizable iff $w=w_0w_0^\shrp$
for some $m$-word $w_0\in\mathbf{D}\backslash\{\ep\}$ and
$f$ is similar to $f^\sharp$ in $\K_\om[x;\tau^{n/m}]$,
where $f^\sharp $ is given by \eqref{eq:2nd_fsharp}.
\end{theorem}

\begin{remark}
 From Theorem \ref{thm:2ndkind_f} it follows that $f^{\sharp\sharp}$ is similar to $f$.
 This is not apparent from \eqref{eq:2nd_fsharp} but by comparing the coefficients of $f$
 and $f^{\sharp\sharp}$ one can verify that
 \[f^{\sharp\sharp}=\{(2d+1)l\}\cdot\tau^{\frac{n}{m}(m+1)}(f)\cdot \{l\}^{-1}.\]
 Using that $\tau^{n/m}(f)$ is similar to $f$ in $\K_\om[x;\tau^{n/m}]$ we conclude that
 indeed $f^{\sharp\sharp}\sim f$.
\end{remark}

\begin{proof}[Proof of Theorem \ref{thm:2ndkind_f}.]
Let $z_1z_2\cdots z_n=w$.
It will also be convenient to define $z_j=z_i$ when $j\equiv i\;\text{(mod $n$)}$.
Assume for a moment that we have proved \eqref{eq:2nd_fsharp} for the case
$z_1=x$ and suppose that $z_1=y$. 
By the shift isomorphism \eqref{eq:DGOiso2},
which holds also for decomposable $f$,
we have
\begin{equation} \label{eq:2nd_yfromx1}
V\simeq V(\om,w(-lm),\tau^{-l}(f))=V(\om,w_0^\shrp w_0,\tau^{-l} (f))
\end{equation}
where $\tau^{-l} (f)=\tau^{-l} (\al_0)
+\tau^{-l} (\al_1)x+\cdots + \tau^{-l} (\al_d) x^d$.
By the assumption we then have
\begin{equation} \label{eq:2nd_yfromx2}
V(\om,w_0^\shrp w_0, \tau^{-l} (f))^\sharp \simeq V(\om,w_0^\shrp w_0, g),
\end{equation}
where 
\[g=\sum_{k=0}^d \{2lk\}\cdot \tau^{(2k+1)l} \Big(
 \overline{ \tau^{-l} \big( \al_{d-k} \big) } \Big) \cdot x^k =
\sum_{k=0}^d \tau^{-l}\Big( \tau^{l}\big( \{2lk\} \big)\cdot \tau^{(2k+1)l}
 \big( \overline{\al_{d-k}} \big)\Big) \cdot x^k.
\]
Again by \eqref{eq:DGOiso2},
\begin{equation} \label{eq:2nd_yfromx3}
V(\om,w_0^\shrp w_0,g)\simeq V(\om,w_0w_0^\shrp,\tau^{l}(g)).
\end{equation}
From the formula
\[\tau^{l}\big( \{ 2lk\}\big)=\{l\}^{-1}\cdot \{ 2lk\} \cdot \tau^{2lk}\big( \{l \}\big)\]
we see that $\tau^{l}(g) = \{ l \}^{-1} \cdot f^\sharp \cdot \{ l \}$ which is
similar to $f^\sharp$. Combining this fact with the isomorphisms
\eqref{eq:2nd_yfromx1}-\eqref{eq:2nd_yfromx3} we deduce
that $V^\sharp\simeq V(\om,w,f^\sharp)$.
Therefore the case $z_1=y$ follows from the case $z_1=x$.

Thus we assume for the rest of the proof that $z_1=x$.
\flushleft\textbf{Step 1.}
Put $a_k=\al_{k-1}/\al_d$ for $k=1,2,\ldots,d$.
Let us replace $f$ by $(1/\al_d)f=a_1+a_2x+\cdots+a_d x^{d-1} + x^d$.
This does not change the isomorphism class of the module $V$.
As in the proof of Theorem \ref{thm:2ndkind_w},
we can construct a basis $[\mf,f_{ks}]$ for $V^{\sharp}$ such that
\begin{align}\label{eq:2pf_phiiso}
\varphi : V(\om,w_0w_0^\shrp,g) &\to V^\sharp \\
 [\mf,e_{ks}] &\mapsto [\mf,f_{ks}] \nonumber 
\end{align}
is an $A$-module isomorphism for some $g$.
We use the decomposition \eqref{eq:Vsharpdecomp}.
We put also $(V^\sharp)_\mf^{(l)}=(V^\sharp)_\mf^{(k)}$ whenever
 $l\in\Z, l\equiv k \;\text{(mod $n$)}.$
By relation \eqref{eq:2ndkind_dualY}, which holds in $V^\sharp$ since $z_1=x$,
it follows that if $1\le k\le n$ and $z_k=y$,
so that $z_{lm+k}=z_k^\shrp=x$,
then 
\[Y:(V^\sharp)_{\si(\mf_{k-1})}^{(lm+k)}\to (V^\sharp)_{\mf_{k-1}}^{(lm+k-1)}\]
is bijective. For the case $k=lm+1$ it is essential that $a_1\neq 0$.
Put
\begin{equation}\label{eq:2ndkind_fdef1_f}
 [\si(\mf_0),f_{11}] = [\si(\mf_0),e_{lm+1,1}^\sharp]
\end{equation}
and recursively
\begin{equation}\label{eq:2ndkind_fdef2_f}
 [\mf,f_{ks}] =\begin{cases}
\si(t)_\mf^{-1} X[\si^{-1}(\mf),f_{ks}],
  & \si^{-1}(\mf)\notin B_\om,\\
X[\si^{-1}(\mf),f_{k-1,s}],
  &\si^{-1}(\mf)\in B_\om, k>1, z_k=x,\\
\big( Y|_{(V^\sharp)_\mf^{(k+lm)}}\big)^{-1} [\si^{-1}(\mf),f_{k-1,s}],
  &\si^{-1}(\mf)\in B_\om, z_k=y, (k>1),\\
X[\si^{-1}(\mf),f_{1,s-1}],
  &\si^{-1}(\mf)\in B_\om, k=1, (z_1=x).
\end{cases}
\end{equation}
By induction, $[\mf,f_{ks}]\in (V^\sharp)_{\mf}^{(lm+k)}$ for each $\mf\in\om$, $s=1,\ldots,d$,
 $k=1,\ldots,n$, $k\equiv j(\mf)\;\text{(mod $m$)}$.

\flushleft\textbf{Step 2.}
We will now show that the $g$ such that $V(\om,w_0w_0^\shrp,g)\simeq V^\sharp$,
is similar to $f^\sharp$, given by \eqref{eq:2nd_fsharp}.
Define an operator $Z:(V^\sharp)_{\mf_0}^{(lm)}\to (V^\sharp)_{\mf_0}^{(lm)}$ by
\begin{equation}
Z=Z_n\cdots Z_2Z_1,
\end{equation}
where $Z_i:(V^\sharp)_{\mf_{i-1}}^{(lm+i-1)}\to (V^\sharp)_{\mf_i}^{(lm+i)}$ are given by
\begin{equation}
Z_i=\begin{cases}
(t_i)^{-1}X^{p_i},&\text{if $z_i=x$,}\\
(t_i)^{-1}X^{p_i-1}\big(Y|_{(V^\sharp)_{\si(\mf_{i-1})}^{(lm+i)}}\big)^{-1},&\text{if $z_i=y$.}
\end{cases}
\end{equation}
Recall that $\mf_0,\mf_1,\ldots,\mf_{m-1}$ are the breaks in $\om$,
ordered such that $\mf_{i-1}<\mf_i<\mf_{i+1}$ for $0< i < m-1$.
The weight diagram in Figure \ref{fig:Vtyp2} can be used to visualize
the action of the operator $Z$.
For an interpretation of the operator $Z$, see also Remark \ref{rem:interpretZ}.
It has the following properties:
\begin{align}
\label{eq:2nd_Z2}
Z[\mf_0,e_{lm,1}^\sharp] &= [\mf_0,f_{n1}],\\
\label{eq:2nd_Zs}
[\mf_0,f_{ns}] &= Z^{s-1}[\mf_0,f_{n1}],\qquad\text{for $s=1,2,\ldots,d$.}
\end{align}
Let us prove \eqref{eq:2nd_Z2}. First
we prove that
\begin{equation}\label{eq:2nd_Z2step1}
Z_1[\mf_0,e_{lm,1}^\sharp]=[\mf_1,f_{11}].
\end{equation}
Since $z_1=x$, using
relation \eqref{eq:2ndkind_dualX} and that $z_{lm+1}=z_1^\shrp=y$,
we have
\begin{equation}\label{eq:2nd_Z2step987}
Z_1[\mf_0,e_{lm,1}^\sharp]=(t_1)^{-1}X^{p_1}[\mf_0,e_{lm}^\sharp]=
(t_1)^{-1}X^{p_1-1}[\si(\mf_0),e_{lm+1,1}^\sharp].
\end{equation}
By definition \eqref{eq:tidef} of $t_1$ and of the vector
$[\si(\mf_0),f_{11}]$, the right hand side of \eqref{eq:2nd_Z2step987} is equal to
\begin{equation}\label{eq:2nd_Z2step988}
\big(\si(t)\si^2(t)\cdots\si^{p_1-1}(t)\big)_{\mf_1}^{-1}X^{p_1-1}[\si(\mf_0),f_{11}].
\end{equation}
Using that $\si(r)_{\si(\mf)}Xv=Xr_\mf v$ for any weight vector $v$
of weight $\mf$ and any $r\in R$,
where $r_\mf$ denotes $r+\mf\in R/\mf$ as usual,
the expression \eqref{eq:2nd_Z2step988} can be rearranged into
(recall that $\si^{p_1}(\mf_0)=\mf_1$)
\begin{equation}\label{eq:2nd_Z2step989}
\big(\si(t)_{\si^{p_1}(\mf_0)}^{-1}X\big)
\big(\si(t)_{\si^{p_1-1}(\mf_0)}^{-1}X\big)\cdots
\big(\si(t)_{\si^2(\mf_0)}^{-1}X\big)[\si(\mf_0),f_{11}].
\end{equation}
By the recursive definition, \eqref{eq:2ndkind_fdef2_f}, the expression in 
\eqref{eq:2nd_Z2step989} is equal to $[\si^{p_1}(\mf_0),f_{11}]=[\mf_1,f_{11}]$,
proving \eqref{eq:2nd_Z2step1}.
Similarly one proves that
\[Z_k[\mf_{k-1},f_{k-1,1}]=[\mf_k,f_{k1}]
\quad\text{for $k=2,3,\ldots,n$.}\]
Combining this with \eqref{eq:2nd_Z2step1},
\eqref{eq:2nd_Z2} is proved.

In the same way one shows that $[\mf_0,f_{ns}]=Z[\mf_0,f_{n,s-1}]$
for $s=2,3,\ldots,d$. Then \eqref{eq:2nd_Zs} follows.

\flushleft\textbf{Step 3.}
We have
\begin{equation}
\label{eq:2nd_Z3}
Z[\mf_0,e_{lm,s}^\sharp] =
\begin{cases}
\{ 2l\}^{-1}\cdot
\big(-\tau^l(\overline{a_{s+1}}/\overline{a_1})[\mf_0,e_{lm,1}^\sharp ]+[\mf_0,e_{lm,s+1}^\sharp ]\big),&
\text{if $s<d$,}\\
-\{ 2l\}^{-1}\tau^l(1/\overline{a_1}) [\mf_0,e_{lm,1}^\sharp ],&\text{if $s=d$.}
\end{cases}
\end{equation}
To prove this, we first prove that if $1\le k\le lm$, so that $lm+k-1<n$, then
\begin{equation}\label{eq:2nd_delsteg}
Z_k [\mf_{k-1},e_{lm+k-1,s}^\sharp ]= (t_k)^{-1} [\mf_k, e_{lm+k,s}^\sharp ]
\end{equation}
for any $1\le s\le d$.
Indeed, if $z_k=x$, then
\begin{align*}
Z_k&[\mf_{k-1},e_{lm+k-1,s}^\sharp ]=\\
 &=(t_k)^{-1} X^{p_k}[\mf_{k-1},e_{lm+k-1,s}^\sharp ]
 &&\text{by definition of $Z_k$,}\\
&=(t_k)^{-1} X^{p_k-1}[\si(\mf_{k-1}),e_{lm+k,s}^\sharp ]
 &&\text{by \eqref{eq:2ndkind_dualX}, since $z_{lm+k}=z_k^\shrp=y$,}\\
&=(t_k)^{-1}[\mf_k,e_{lm+k,s}^\sharp ],
 &&\text{by first case in \eqref{eq:2ndkind_dualX}.}
\end{align*}
We used that $\si^{p_k}(\mf_{k-1})=\mf_k$ in the last step.
Similarly, if $z_k=y$, then 
\[ Y[\si(\mf_{k-1}),e_{lm+k,s}^\sharp ]=[\mf_{k-1},e_{lm+k-1,s}^\sharp ]\]
by \eqref{eq:2ndkind_dualY} since $z_{lm+k}=z_k^\shrp=x$ and
$1<lm+k\le n$. Therefore
\[ \big(Y|_{(V^\sharp )_{\si(\mf_{k-1})}^{(lm+k)}}\big)^{-1}
[\mf_{k-1},e_{lm+k-1,s}^\sharp ]=
[\si(\mf_{k-1}),e_{lm+k,s}^\sharp ]\]
and
\begin{align*}
Z_k[\mf_{k-1},e_{lm+k-1,s}^\sharp ]&=(t_k)^{-1} X^{p_k-1}
 \big(Y_{(V^\sharp )_{\si(\mf_{k-1})}^{(lm+k)}}\big)^{-1}
[\mf_{k-1},e_{lm+k-1,s}^\sharp ]=\\
&=(t_k)^{-1}X^{p_k-1}[\si(\mf_{k-1}),e_{lm+k,s}^\sharp ]=\\
&=(t_k)^{-1}[\mf_k,e_{lm+k,s}^\sharp ].
\end{align*}
This proves \eqref{eq:2nd_delsteg}.

Using \eqref{eq:2nd_delsteg} repeatedly for $k=1,2,\ldots ,lm$ while
moving the $t_i$'s to the left, we have
\begin{align*}
&Z_mZ_{m-1}\cdots Z_2Z_1[\mf_0,e_{lm,s}^\sharp ]=\\
&=Z_mZ_{m-1}\cdots Z_2
\cdot (t_1)^{-1}[\mf_1,e_{lm+1,s}^\sharp ]=\\
&=\si^{p_2+p_3+\cdots +p_m}(t_1)^{-1} Z_mZ_{m-1}\cdots Z_2 [\mf_1,e_{lm+1,s}^\sharp ]
=\cdots =\\
&=\si^{p_2+p_3+\cdots +p_m}(t_1)^{-1}
\si^{p_3+p_4+\cdots +p_m}(t_2)^{-1}\cdots \si^{p_m}(t_{m-1})^{-1}\cdot (t_m)^{-1}\cdot\\
&\quad \cdot [\mf_m,e_{lm+m,s}^\sharp ] =\\
&= q^{-1}\cdot [\mf_0,e_{(l+1)m,s}^\sharp ].
\end{align*}
Here we used that, from the definition of $Z_k$, we have
 $Z_k\la v=\si^{p_k}(\la) Z_k v$ for any $k\in\{1,\ldots,n\}$
 (recall that, by convention, $p_{km+j}=p_j$ if $1\le j\le m$)
 any $\la\in R/\mf$ and any
 weight vector $v$ of weight $\mf$, where $\si$ here denotes the map
$R/\mf\to R/\si(\mf)$ induced by $\si$.
We would like to continue, acting by $Z_{m+1}$, then $Z_{m+2}$ and so on
up to $Z_{lm}$. First we need to move the $q^{-1}$ to the left. 
For any $k\in\{1,\ldots,l\}$ we have
 $Z_{km}Z_{km-1}\cdots Z_{(k-1)m+1}\la v=\tau(\la) Z_{km}Z_{km-1}\cdots Z_{(k-1)m+1} v$
since $\tau=\si^p$ and $p=p_1+p_2+\cdots p_m$. Therefore,
using \eqref{eq:2nd_delsteg} as in the above calculation we get
\begin{align}
Z_{lm}Z_{lm-1}\cdots Z_1 [\mf_0,e_{lm,s}^\sharp ] &= \nonumber 
 Z_{lm}Z_{lm-1}\cdots Z_{m+1}\cdot q^{-1}[\mf_0,e_{(l+1)m,s}^\sharp ]= \nonumber \\
 &=\tau^{l-1}(q^{-1})Z_{lm}Z_{lm-1}\cdots Z_{m+1} [\mf_0,e_{(l+1)m,s}^\sharp ]=\nonumber \\
 &\ldots\nonumber \\
 &=\tau^{l-1}(q^{-1})\tau^{l-2}(q^{-1})\cdots \tau(q^{-1})q^{-1}\cdot [\mf_0,e_{2lm,s}^\sharp ]=\nonumber \\
 &=\{ l\}^{-1}\cdot [\mf_0,e_{n,s}^\sharp ]. \label{eq:2nd_delstegA}
\end{align}
It remains to calculate $Z_{2lm}Z_{2lm-1}\cdots Z_{lm+1}[\mf_0,e_{n,s}^\sharp ]$.
To calculate $Z_{lm+1}[\mf_0,e_{n,s}^\sharp ]$ we need to find, by definition
of $Z_{lm+1}$,
\[\big(Y|_{(V^\sharp )_{\si(\mf_0)}^{(1)}}\big)^{-1} [\mf_0,e_{ns}^\sharp ] \]
because $z_{lm+1}=z_1^\shrp = y$. By \eqref{eq:2ndkind_dualY},
\begin{align}
Y[\si(\mf_0),e_{1,s+1}^\sharp ] &=[\mf_0,e_{n,s}^\sharp ]- \overline{ a_{s+1} }
 \cdot [\mf_0, e_{n,d}^\sharp ],\quad\text{if $s<d$,}\\
Y[\si(\mf_0),e_{1,1}^\sharp ] &=-\overline{a_1}\cdot [\mf_0, e_{n,d}^\sharp ].
\end{align}
Therefore
\begin{multline}\label{eq:2nd_delsteg2}
\big( Y|_{(V^\sharp)_{\si(\mf_0)}^{(1)}}\big) ^{-1} [\mf_0,e_{n,s}^\sharp ]=\\=
\begin{cases}
[\si(\mf_0),e_{1,s+1}^\sharp ] - \si(\overline{ a_{s+1} }/\overline{ a_1} )\cdot
 [\si(\mf_0),e_{1,1}^\sharp ], & s<d,\\
-\si(1/\overline{ a_1 })\cdot [\si(\mf_0),e_{1,1}^\sharp ],& s=d.
\end{cases}
\end{multline}
Applying $(t_1)^{-1}X^{p_1-1}$ to both sides of \eqref{eq:2nd_delsteg2}
we deduce that
\begin{equation}
Z_{lm+1}[\mf_0,e_{n,s}^\sharp ]=
(t_1)^{-1}\cdot
\begin{cases}
[\mf_1,e_{1,s+1}^\sharp ] - \si^{p_1}(\overline{a_{s+1}}/\overline{a_1} )\cdot
 [\mf_1,e_{1,1}^\sharp ], & s<d,\\
-\si^{p_1}(1/\overline{a_1})\cdot [\mf_1,e_{1,1}^\sharp ],& s=d.
\end{cases}
\end{equation}
Similarly to relation \eqref{eq:2nd_delsteg}, we have the formula
\begin{equation}\label{eq:2nd_delsteg3}
Z_{lm+k}[\mf_{k-1},e_{k-1,s}^\sharp ]=(t_k)^{-1} [\mf_k, e_{k,s}^\sharp ]
\quad\text{for $1<k\le lm$ and $1\le s\le d$,}
\end{equation}
which can be proved using \eqref{eq:2ndkind_dualX}, \eqref{eq:2ndkind_dualY}.
Note that $t_{lm+k}=t_k$ by the notational assumptions on $\mf_k$ and $t_k$.
Using \eqref{eq:2nd_delsteg3} repeatedly we get
\begin{multline}
Z_{(l+1)m}Z_{(l+1)m-1}\cdots Z_{lm+1}[\mf_0,e_{n,s}^\sharp ]=\\
=q^{-1}\cdot
\begin{cases}
[\mf_0,e_{m,s+1}^\sharp ] - \tau(\overline{a_{s+1}}/\overline{a_1} )\cdot
 [\mf_0,e_{m,1}^\sharp ], & s<d,\\
-\tau(1/\overline{a_1})\cdot [\mf_0,e_{m,1}^\sharp ],& s=d.
\end{cases}
\end{multline}
Repeating we get
\begin{multline}\label{eq:2nd_delstegB}
Z_{2lm}Z_{2lm-1}\cdots Z_{lm+1}[\mf_0,e_{n,s}^\sharp ]=\\
=\{l\}^{-1}\cdot
\begin{cases}
[\mf_0,e_{lm,s+1}^\sharp ] - \tau^l(\overline{a_{s+1}}/\overline{ a_1} )\cdot
 [\mf_0,e_{lm,1}^\sharp ], & s<d,\\
-\tau^l(1/\overline{a_1})\cdot [\mf_0,e_{lm,1}^\sharp ],& s=d.
\end{cases}
\end{multline}
Thus, combining \eqref{eq:2nd_delstegA} and \eqref{eq:2nd_delstegB}
we obtain \eqref{eq:2nd_Z3} as desired.

\flushleft\textbf{Step 4.}
Set $b_s=-\overline{a_s}/\overline{a_1}$ for $2\le s\le d$ and $b_1=-1/\overline{a_1}$.
We claim that for $1\le s<d$, there are constants $C_{s1},C_{s2},\ldots,C_{ss}\in\K_\om$
such that
\begin{multline} \label{eq:2nd_ind}
[\mf_0,f_{ns}]=C_{s1}\tau^{3l}(b_s)[\mf_0,f_{n1}]
 +\cdots + C_{s,s-1}\tau^{l+2l(s-1)}(b_2)[\mf_0,f_{n,s-1}]+\\
 +C_{s,s}\big(\tau^l(b_{s+1})[\mf_0,e_{lm,1}^\sharp]+[\mf_0,e_{lm,s+1}^\sharp]\big)
\end{multline}
We prove this by induction on $s$.
If $s=1$ we can take
\begin{equation}\label{eq:2nd_C11}
C_{11}=\{2l\}^{-1}
\end{equation}
by \eqref{eq:2nd_Z2} and \eqref{eq:2nd_Z3}.
Assume \eqref{eq:2nd_ind} holds for some $s<d-1$. Then, using
\eqref{eq:2nd_Zs} and that $Z\la=\tau^{2l}(\la)Z$ for any $\la\in\K_{\mf_0}$, we have
\begin{align*}
[&\mf_0,f_{n,s+1}]=Z[\mf_0,f_{ns}]=\\
&=
\tau^{2l}(C_{s1})\tau^{5l}(b_s)Z[\mf_0,f_{n1}]+\cdots+
 \tau^{2l}(C_{s,s-1})\tau^{l+2ls}(b_2)Z[\mf_0,f_{n,s-1}]+\\
&\quad +\tau^{2l}(C_{s,s})\big(\tau^{3l}(b_{s+1})Z[\mf_0,e_{lm,1}^\sharp]
+Z[\mf_0,e_{lm,s+1}^\sharp]\big)
\end{align*}
By \eqref{eq:2nd_Z2},\eqref{eq:2nd_Zs} and \eqref{eq:2nd_Z3} this equals
\begin{align*}
&
\tau^{2l}(C_{s,s})\tau^{3l}(b_{s+1})[\mf_0,f_{n1}]+\\&\quad+
\tau^{2l}(C_{s1})\tau^{5l}(b_s)[\mf_0,f_{n2}]+\cdots+
 \tau^{2l}(C_{s,s-1})\tau^{l+2ls}(b_2)[\mf_0,f_{n,s}]+\\
&\quad +\tau^{2l}(C_{s,s})\{2l\}^{-1}\cdot 
 \big(\tau^l(b_{s+2})[\mf_0,e_{lm,1}^\sharp]+[\mf_0,e_{lm,s+2}^\sharp]\big).
\end{align*}
Thus we seek the solution to the following system of equations
\begin{align}\label{eq:2nd_C0}
C_{s+1,1}&=\tau^{2l}(C_{s,s}),\\
\label{eq:2nd_C1}
C_{s+1,r}&=\tau^{2l}(C_{s,r-1}),\; 2\le r\le s,\\
\label{eq:2nd_C2}
C_{s+1,s+1}&=\tau^{2l}(C_{s,s}) \{2l\}^{-1}.
\end{align}
From \eqref{eq:2nd_C2},\eqref{eq:2nd_C11} we deduce
\begin{equation}\label{eq:2nd_Css}
C_{s,s}= \{2ls\}^{-1}
\qquad 1\le s <d.
\end{equation}
Repeated use of \eqref{eq:2nd_C1} gives
For $1\le r< s<d$ we have
\begin{align*}
C_{s,r}&=\tau^{2l}(C_{s-1,r-1})=\cdots=\tau^{2l(r-1)}(C_{s-r+1,1})
 &\text{by \eqref{eq:2nd_C1}}\\
&=\tau^{2lr}(C_{s-r,s-r})
 &\text{by \eqref{eq:2nd_C0}}\\
&= \{2lr\} \{2ls\}^{-1}
 &\text{by \eqref{eq:2nd_Css}.}
\end{align*}
Substituting this and \eqref{eq:2nd_Css} into \eqref{eq:2nd_ind} we obtain
that, for $1\le s<d$,
\begin{align}
[\mf_0,f_{ns}]&=
\{2l\}\{2ls\}^{-1}
 \cdot\tau^{3l}(b_s)\cdot [\mf_0,f_{n1}]+\nonumber \\
&+
\{4l\}\{2ls\}^{-1}
 \cdot\tau^{5l}(b_{s-1})\cdot [\mf_0,f_{n2}]+\nonumber \\
&\quad\cdots\nonumber \\
&+
\{2l(s-1)\}\{2ls\}^{-1}
 \cdot\tau^{l+2l(s-1)}(b_2)
 \cdot [\mf_0,f_{n,s-1}]+\nonumber \\
&+
\{2ls\}^{-1}
\big(\tau^l(b_{s+1})[\mf_0,e_{lm,1}^\sharp]+[\mf_0,e_{lm,s+1}^\sharp]\big).
\label{eq:2nd_fns}
\end{align}
In particular, taking $s=d-1$ and applying $Z$ we have
\begin{align*}
[\mf_0,f_{nd}]&=Z[\mf_0,f_{n,d-1}]=\\
&=\{4l\}\{2ld\}^{-1}
\cdot\tau^{5l}(b_{d-1})\cdot[\mf_0,f_{n2}]+\nonumber\\
&+
\{6l\}\{2ld\}^{-1}
\cdot\tau^{7l}(b_{d-2})\cdot[\mf_0,f_{n3}]+\nonumber\\
&\quad \cdots \nonumber \\
&+
\{2l(d-1)\}\{2ld\}^{-1}
\cdot\tau^{l+2l(d-1)}(b_2)\cdot[\mf_0,f_{n,d-1}]+\nonumber\\
&+
\{2l\}\{2ld\}^{-1}
\cdot
\big(\tau^{3l}(b_d)[\mf_0,f_{n1}]+
\{2l\}^{-1}
\tau^l(b_1)
[\mf_0,e_{lm,1}^\sharp]\big),
\end{align*}
where we applied \eqref{eq:2nd_Z3} in the last term.
Hence,
using that
\[X[\mf_0,e_{lm,1}^\sharp]=[\si(\mf_0),f_{11}]=
[\si(\mf_0),e_{lm+1,1}^\sharp ] \]
by \eqref{eq:2ndkind_dualX} and that $z_{lm+1}=z_1^\shrp=y$,
together with the relation (recall $\varphi$ from \eqref{eq:2pf_phiiso})
\begin{align*}
X[\mf_0,f_{ns}]&=X\varphi\big([\mf_0,e_{ns} ]\big)=
\varphi\big(X[\mf_0,e_{ns} ]\big)=\\
 &=\varphi\big([\si(\mf_0),e_{1,s+1} ]\big)= [\si(\mf_0),f_{1,s+1}]
\end{align*}
holding for $s<d$, we obtain that
\begin{align*}
X[\mf_0,f_{nd}]&=
\si\big( \{2ld\}^{-1}\tau^l(b_1)\big)\cdot [\si(\mf_0),f_{11}]+\nonumber\\
&+\si\big( \{2l\}\{2ld\}^{-1}\tau^{3l}(b_d)\big)\cdot [\si(\mf_0),f_{12}]+\nonumber\\
&+\si\big( \{4l\}\{2ld\}^{-1}\tau^{5l}(b_{d-1})\big)\cdot[\si(\mf_0),f_{13}]+\nonumber\\
&\quad\cdots\nonumber\\
&+\si\big( \{2l(d-1)\}\{2ld\}^{-1} \tau^{l+2l(d-1)}(b_2)\big)\cdot[\si(\mf_0),f_{1d}].
\end{align*}
Resubstituting $b_1=-1/\overline{a_1}=-\overline{\al_d}/\overline{\al_0}$
and $b_s=-\overline{a_s}/\overline{a_1}=-\overline{\al_{s-1}}/\overline{\al_0}$
(for $s>1$),
we conclude that, in view of the final case in relation \eqref{eq:2ndkindXpf},
the map $V(\om,w_0w_0^\shrp,g)\to V^\sharp$, $[\mf,e_{ks}]\mapsto [\mf,f_{ks}]$
will be an $A$-module isomorphism if $g$ is given by
\begin{align*}
\{2ld\}\cdot g &=
\tau^l(\overline{\al_d}/\overline{\al_0}) +\nonumber\\
&+\{2l\}\cdot \tau^{3l}(\overline{\al_{d-1}}/ \overline{ \al_0} ) \cdot x+\nonumber\\
&+\{4l\}\cdot \tau^{5l}(\overline{\al_{d-2}}/\overline{\al_0} ) \cdot x^2+\nonumber\\
&\quad\cdots\nonumber\\
&+\{2l(d-1)\}\cdot \tau^{l+2l(d-1)}(\overline{\al_1}/\overline{\al_0} ) \cdot x^{d-1}+\nonumber\\
&+\{2ld\}\cdot x^d.
\end{align*}
Thus $\{2ld\}\cdot g\cdot \tau^l(\overline{\al_0}) = f^\sharp $
so $g$ is similar to $f^\sharp$. This finishes the proof
that $V^\sharp\simeq V(\om,w_0w_0^\shrp,f^\sharp)$.
\end{proof}

\begin{remark}\label{rem:interpretZ}
The indecomposable weight
module $V=V(\om,w,f)$, $w=z_1\cdots z_n$, has the the following
characterizing properties:
\begin{enumerate}
\item[1)] the operator $Z=Z(w):V_{\mf_0}\to V_{\mf_0}$ given by $Z=Z_n\cdots Z_2Z_1$
where
\[Z_i=\begin{cases}(t_i)^{-1}X^{p_i}, & z_i=x,\\
(t_i)^{-1}X^{p_i-1}Y^{-1},& z_i=y,\end{cases}\]
is well-defined and single-valued (since $w$ is non-periodic), and
\item[2)] giving $V_{\mf_0}$ the structure of a module over $\K_\om[x;\tau^{n/m}]$ by
\[ x . v = Zv ,\quad v\in V_{\mf_0},\]
there exists a nonzero vector in $V_{\mf_0}$ which is annihilated by $f$.
\end{enumerate}
What we prove in Theorem \ref{thm:2ndkind_w}
is that $Z(w^\shrp)$ is well-defined on the $\mf_0$-weight space of
$V(\om,w,f)^\sharp$, while
in Theorem \ref{thm:2ndkind_f} we prove that when $V=V(\om,w_0w_0^\shrp,f)$,
the space $(V^\sharp)_{\mf_0}$ contains a nonzero vector annihilated by 
a skew polynomial similar to $f^\sharp$.
Therefore $V^\sharp\simeq V(\om,w_0w_0^\shrp,f^\sharp)$.
\end{remark}

\section{Examples} \label{sec:examples}
\subsection{Noncommutative type-A Kleinian singularities}
Let $R=\C[H]$
and $\si\in\Aut_\C(H)$ be given by $\si(H)=H-1$ and $t\in R$ be arbitrary.
The generalized Weyl algebra $A=R(\si,t)$ was studied in \cite{B} and \cite{H}.
For example, all simple modules (not only weight modules)
were classified in \cite{B}.
Let $\ast$ be the $\mathbb{R}$-algebra automorphism of $R$ given by
$i^\ast=-i$, $H^\ast=H$.
Suppose that $t^\ast=t$ i.e. that $t=f(H)$, where the polynomial $f$ has real coefficients.
Since any orbit is infinite,
Theorem \ref{thm:inforbnobr} and Theorem \ref{thm:inforbwbr}
implies that
an indecomposable weight module with real support
is pseudo-unitarizable iff it is simple.

\subsection{The enveloping algebra of $\mathfrak{sl}_2$} \label{sec:ex_usl2}
Let $R=\C[h,t]$ and let $\si\in\Aut_\C(R)$ be given by $\si(h)=h-2$,
 $\si(t)=t+h$. Then $A=R(\si,t)\simeq U(\mathfrak{sl}_2)$.
Define $\ast\in\Aut_{\mathbb{R}}(R)$ by $h^\ast=h, t^\ast=t, i^\ast=-i$.
Here, as in the previous example, all orbits are infinite
so indecomposable weight modules with real support are
pseudo-unitarizable iff they are simple.

By induction one checks that $\si^n(t)=-n^2+(h+1)n+t,\;\forall n\in\Z$.
Thus, for any $\mu,\al\in\mathbb{R}$,
 \[\lim_{n\to\pm\infty} \big\{ \si^n(t)\;\text{mod}\; (h-\mu,t-\al)\big\}=
 \lim_{n\to\pm\infty} -n^2+(\mu+1)n+\al = -\infty.\]
In view of formulas \eqref{eq:Phila},\eqref{eq:Phila2},\eqref{eq:Phila3},
this shows that any non-degenerate symmetric admissible form on an infinite-dimensional
simple weight module with real support is necessarily indefinite.

On the other hand, on a finite-dimensional simple weight module $V(N)$
(with highest weight $N\in\Z_{\ge 0}$ and of dimension $N+1$), the form
 $\Psi_\la$ given by \eqref{eq:Phila2} with $\la>0$
is positive definite because
\[\si^n(t) \;\text{mod}\; (t,h-N) = n(N-n+1)>0 \]
for $n=1,2,\ldots,N$ so that
$\Psi_\la(Y^ne_0,Y^ne_0)>0$ for $n=0,1,\ldots,N$.

\subsection{The quantum enveloping algebra of $\mathfrak{sl}_2$} \label{sec:ex_uqsl2}
Let $R=\C[K,K^{-1},t]$ and $q\in\C\backslash\{-1,0,1\}$.
Define $\si\in\Aut_\C(R)$ by $\si(K)=q^{-2}K, \si(t)=t+\frac{K-K^{-1}}{q-q^{-1}}$.
Then $R(\si,t)\simeq U_q(\mathfrak{sl}_2)$. We assume here that $q^2$ is
a root of unity of order $p>1$.
Let $\ast\in\Aut_{\mathbb{R}}(R)$ be given by
$K^\ast=K^{-1}$, $i^\ast=-i$, $t^\ast=t$.
One verifies that $\si$ commutes with $\ast$ and that $\si$ has order $p$.
All orbits have $p$ elements and are torsion trivial (recall Definition \ref{dfn:torsion_trivial}).
Let $\om\in\Om$ and $\mf=(K-\mu,t-\al)\in\om$. Then $\om$ is real
iff $\mf^\ast=\mf$ which holds iff $|\mu|=1$ and $\al\in\mathbb{R}$.
Assume $\om$ is real and put $\mf(\om)=\mf$.
We identify $\K_\om=R/\mf$ with $\C$.
The real number
\begin{equation}\label{eq:uqsl2xi}
\xi=\big(\si(t)\si^2(t)\cdots\si^p(t)\big)_\mf=
\prod_{k=0}^{p-1}\Big(\al+\sum_{i=0}^k\frac{q^{-2i}\mu-q^{2i}\mu^{-1}}{q-q^{-1}}\Big)
\end{equation}
is nonzero iff there are no breaks in $\om$. 

Assume that $\xi\neq 0$ and consider the modules $V(\om,f)$.
Since $\si^p=\Id$,
the skew Laurent polynomial ring $\K_\om[x,x^{-1};\tau]$, to which
$f$ belongs, is just 
the ordinary commutative Laurent polynomial ring $P=\C[x,x^{-1}]$.
Similarity in $P$ just means
equality up to multiplication by nonzero homogenous term.
Any indecomposable element in $P$ is similar to $f=(x-a)^d$
for some $a\in\C\backslash\{0\}$, $d\ge 1$. By Theorem \ref{thm:Vomfdual},
$V(\om,f)^\sharp\simeq V(\om,f^\sharp)$ where
$f^\sharp=(\xi x)^d((\xi x)^{-1}-\overline{a})^d = (1-\overline{a}\xi x)^d
\sim (x-(\overline{a}\xi)^{-1})^d$. Thus
we conclude that $V(\om,f)$, where $\om$ is a real orbit without breaks containing
$(K-\mu, t-\al)$ and $f=(x-a)^d$,
is pseudo-unitarizable
iff $a=(\overline{a}\xi)^{-1}$, that is, iff $|a|^2=\xi^{-1}$,
where $\xi$ is given by \eqref{eq:uqsl2xi}.
It would be interesting to determine the values of $\al$ and $\mu$ for which
$\xi$ is positive so that $|a|^2=\xi^{-1}$ can hold.
We only note here that for any fixed $\mu$, the quantity $\xi$ is
a polynomial of degree $p$ in $\al$ with positive leading coefficient
and thus $\xi>0$ if $\al$ is sufficiently big.

Assume now that $\xi=0$. Then $\om$ has breaks and we can assume $\al=0$.
Recall that the break $\mf_0=\mf(\om)=\mf$.
For $k\ge 0$ we have
\[ \si^{k+1}(t) = t+\sum_{i=0}^k
 \frac{q^{-2i}K-q^{2i}K^{-1}}{q-q^{-1}}. \]
Thus the reduction modulo $\mf_0$ is
\begin{equation}\label{eq:uqsl2useful}
\big(\si^{k+1}(t)\big)_{\mf_0} =
\sum_{i=0}^k \frac{q^{-2i}\mu-q^{2i}\mu^{-1}}{q-q^{-1}} =
 \frac{(1-q^{2(k+1)})(1-\mu^2 q^{-2k})}{\mu q(q-q^{-1})^2} 
\end{equation}
This shows that, for $0\le k\le p-2$,
\begin{equation} \label{eq:exuqsl2mu2}
 \si^{-(k+1)}(\mf_0)\in B_\om \Longleftrightarrow \mu^2=q^{2k}.
\end{equation}
By \eqref{eq:exuqsl2mu2} we have
\[
B_\om =
\begin{cases}
\{\mf_0, \mf_1=\si^{-(k+1)}(\mf_0) \}, &\text{if $\mu^2=q^{2k}$ where $0\le k\le p-2$,} \\
\{\mf_0\},& \text{if $\mu\notin\{\pm 1, \pm q, \ldots, \pm q^{p-2}\}$,}
\end{cases}
\]
Call $\mu$ \emph{generic} if 
$\mu\notin\{\pm 1, \pm q, \ldots, \pm q^{p-2}\}$
and \emph{specific} otherwise. If $\mu$ is specific,
we let $r$ ($0\le r\le p-2$) denote the unique integer such that $\mu^2=q^{2r}$.
Let $m=|B_\om|$. By \eqref{eq:exuqsl2mu2}, $m=1$ if $\mu$ is generic
and $m=2$ if $\mu$ is specific.
Recall the definition of $p_i$ from Section \ref{sec:notation}.
For specific $\mu$ we have $p_1=p-(r+1)$ and $p_2=r+1$.

By Theorem \ref{thm:firstkind}, a module of the form $V(\om,j,w)$
is pseudo-unitarizable iff it is simple, which holds
iff $w=\ep$, the empty word.
If $\mu$ is generic then there is only one such module, $V(\om,0,\ep)$.
If $\mu$ is specific then there are two such modules, $V(\om,0,\ep)$ and $V(\om,1,\ep)$.

If $V=V(\om,w=z_1\cdots z_n,f=(x-a)^d)$, then by Theorem \ref{thm:2ndkind}, $V$ 
is pseudo-unitarizable iff $w=w_0w_0^\shrp$ where $w_0$ is a non-empty $m$-word
(so for generic $\mu$ the word $w_0$ is arbitrary, while for specific $\mu$,
it has to be of even length) and $f$ is similar to $f^\sharp$ in $\C[x]$.
Let $(a;s)_i$ denote the shifted factorial
\[(a;s)_i = (1-a)(1-as)\cdots (1-as^{i-1})\]
and for $j<i$ let $(a;s)_i^{(j)}$ denote $(a;s)_i$ but with the factor $(1-as^j)$ omitted.
By \eqref{eq:2nd_fsharp} the polynomial $f^\sharp$ is given by
\[f^\sharp = \sum_{k=0}^d Q^{nk} \overline{\al_{d-k}} \cdot x^k =
 (Q^n x)^d \cdot \overline{f\big((Q^n x)^{-1}\big)} = (1-Q^n\overline{a}x)^d
 \sim \big( x-(Q^n \overline{a})^{-1} \big) ^d,\]
where $Q$ is the nonzero real number given by
\begin{equation}
Q= t_1 =
 \frac{(q^2 ; q^2)_{p-1} \cdot (\mu^2 ; q^{-2})_{p-1}}{(\mu q (q-q^{-1})^2)^{p-1} },
\quad \text{if $\mu$ is generic},
\end{equation}
and
\begin{equation}
Q=\si^{p_2}(t_1)t_2 =
 \frac{(q^2;q^2)_{p-1}^{(r)} \cdot (\mu^2 ; q^{-2})_{p-1}^{(r)}}{(\mu q (q-q^{-1})^2)^{p-2}},
\quad \text{if $\mu$ is specific, $\mu^2=q^{2r}$}.
\end{equation}
We conclude that $V=V(\om,w=z_1\cdots z_n,f=(x-a)^d)$,
($\om$ a real orbit containing a break $\mf=(t,K-\mu)$)
has a non-degenerate admissible form iff $w=w_0w_0^\shrp$,
where $w_0\in\mathbf{D}\backslash\{\ep\}$
has even length if $\mu$ is specific, and $|a|^2=Q^{-n}$.
Since $n$ is even, solutions $a\in\C$ to this equation always exist.

Irreducible representations of $U_q(\mathfrak{sl}_2)$ which are unitarizable
with respect to a positive definite form were described in \cite{V}.
This corresponds to the case when all the factors in \eqref{eq:uqsl2xi}
are nonnegative.

\subsection{When $R$ is a field}
We note that in the special case when $R=\K$ is a field,
there is only one orbit $\om_0$ consiting of the zero ideal alone.
The orbit $\om_0$ is real, and contains a break iff $t=0$. Furthermore,
$\om_0$ is torsion trivial iff $\si$ is trivial. An indecomposable
weight module over $A=R(\si,t)$ is then of the form $V(\om,f)$
if $t\neq 0$, where $f\in\K[x,x^{-1};\si]$ and
$V(\om,j,w)$ or $V(\om,w,f)$ if $t=0$, where
$f\in\K[x;\si^n]$ ($n=|w|$).
This shows that any skew polynomial ring can occur.

\subsection{An example of a module of the second kind} \label{sec:ex_2ndkind}
Let $R=\C[u,t]$, $\si\in\Aut_\C(R)$ defined by $\si(u)=1-u, \si(t)=t$.
Then the orbits have the form
$\om_{\mu,\al}=\{(u-\mu,t-\al), (u-(1-\mu), t-\al)\}$,
where $\mu,\al\in\C$. All orbits are torsion trivial and
have two elements, except for $\om_{1/2,\al}$ which has only one element.
The orbit $\om_{\mu,\al}$ contains no breaks if $\al\neq 0$,
and all elements of $\om_{\mu,0}$ are breaks.
Define $\ast\in\Aut_{\mathbb{R}}(R)$ by $u^\ast=u$, $t^\ast=t$, $i^\ast=-i$.
Then $\om_{\mu,\al}$ is real iff $\mu,\al\in\mathbb{R}$.

Let $\om=\om_{0,0}$.
Let $\mf(\om)=\mf_0=(u,t)$ and $\si(\mf_0)=\mf_1=(u-1,t)$.
Then $B_\om=\om$, $p=|\om|=2$, $m=|B_\om|=2$.
We identify $\K_\om=R/\mf(\om)$ with $\C$. The map $\tau$ is the identity
since $\om$ is torsion trivial.
Let $f=a_1+a_2x+x^2\in\C[x]$, $a_1\neq 0$, let $w=xxyy$ and
let $V=V(\om,w,f)$. The weight module $V$ is decomposable iff $f$
has distinct roots.

Since $\si(\mf_0)=\mf_1$ and $\si(\mf_1)=\mf_0$, the
integers $p_1$ and $p_2$ (defined in Section \ref{sec:notation})
both equal one. Thus, recalling
definitions \eqref{eq:qdef}, \eqref{eq:tidef} of $q$, $t_1,t_2$,
we have $t_1=t_2=1$ and $q=1$.
By Theorem \ref{thm:2ndkind_f}, $V^\sharp\simeq V(\om,w,f^\sharp)$
where $f^\sharp= 1+\overline{a_2}x+\overline{a_1}x^2\sim
 1/\overline{a_1}+\overline{a_2}/\overline{a_1}\cdot x+x^2$.
Thus $V\simeq V^\sharp$ iff $a_1=1/\overline{a_1}$,
$a_2=\overline{a_2}/\overline{a_1}$.

The module $V$ has the following structure.
We have $V=V_{\mf_0}\oplus V_{\mf_1}$.
Since $j(\mf_0)=0$ and $j(\mf_1)=1$,
$V_{\mf_0}$ has a basis $\{e_{21},e_{22},e_{41},e_{42}\}$ and
$V_{\mf_1}$ has a basis $\{e_{11},e_{12},e_{31},e_{32}\}$.
\begin{figure}
\caption{Weight diagram for $V$}
\label{fig:8dim}
\[\xymatrix@R-10pt{
&\bullet \save[]+<0pt,-9pt>*{e_{11}}+<0pt,18pt>*{\mf_1}\restore \ar[r]
&\bullet \save[]+<0pt,-9pt>*{e_{21}}+<0pt,18pt>*{\mf_0}\restore
&\bullet \save[]+<0pt,-9pt>*{e_{31}}+<0pt,18pt>*{\mf_1}\restore \ar[l]
&\bullet \save[]+<9pt,-9pt>*{e_{41}}+<-9pt,18pt>*{\mf_0}\restore \ar[l]
 \ar@(d,u)[ddlll]_X \\
&\\
&\bullet \save[]+<0pt,-9pt>*{e_{12}} \restore \ar[r]
&\bullet \save[]+<0pt,-9pt>*{e_{22}} \restore
&\bullet \save[]+<0pt,-9pt>*{e_{32}} \restore \ar[l]
&\bullet \save[]+<9pt,-9pt>*{e_{42}} \restore \ar[l]
 \ar@<-0pt> `d[llllu] `[llllu] `[lllu]^X [lllu]
\save "2,2"."1,2"."3,2"*[F]\frm{}
\restore
}\]
\end{figure}
See Figure \ref{fig:8dim}.
The module structure on $V$ is given by the following, where $s=1,2$:
\begin{align*}
\begin{cases}
Xe_{1s}=e_{2s},\\ 
Xe_{2s}=Xe_{3s}=0,\\ 
Xe_{41}=e_{12},&\\
Xe_{42}=-a_1e_{11}-a_2e_{12},&
\end{cases}
&&
\begin{cases}
Ye_{1s}=0,\\
Ye_{2s}=0,\\ 
Ye_{3s}=e_{2s},\\ 
Ye_{4s}=e_{3s}.\\ 
\end{cases}
\end{align*}
Let us show explicitly that $V^\sharp\simeq V(\om,w,f^\sharp)$.
Let $\{e_{ks}^\sharp : 1\le k\le 4, s=1,2\}$ be the dual basis in $V^\sharp$, i.e.
 $e_{ks}^\sharp(e_{ij})=\delta_{ki}\delta_{sj}$.
Then $\{e_{2s}^\sharp,e_{4s}^\sharp : s=1,2\}$ is a basis for $(V^\sharp)_{\mf_0}$
and $\{e_{1s}^\sharp, e_{3s}^\sharp : s=1,2\}$ is a basis for $(V^\sharp)_{\mf_1}$.
For $s=1,2$ we have
\begin{align*}
\begin{cases}
Xe_{1s}^\sharp = 0,\\
Xe_{2s}^\sharp = e_{3s}^\sharp,\\
Xe_{3s}^\sharp = e_{4s}^\sharp,\\
Xe_{4s}^\sharp = 0,
\end{cases}
&&
\begin{cases}
Ye_{11}^\sharp = - \overline{a_1} e_{42}^\sharp ,\\
Ye_{12}^\sharp = e_{41}^\sharp - \overline{a_2} e_{42}^\sharp ,\\
Ye_{2s}^\sharp = e_{1s}^\sharp ,\\
Ye_{3s}^\sharp = Ye_{4s}^\sharp = 0.
\end{cases}
\end{align*}
Set $b_1=-1/\overline{a_1}$ and $b_2=-\overline{a_2}/\overline{a_1}$ and
\begin{align}\label{eq:ex_f38}
\begin{cases}
f_{11}=e_{31}^\sharp ,\\
f_{21}=e_{41}^\sharp ,\\
f_{31}=b_2e_{11}^\sharp+e_{12}^\sharp ,\\
f_{41}=b_2e_{21}^\sharp+e_{22}^\sharp ,
\end{cases}
&&
\begin{cases}
f_{12}=b_2e_{31}^\sharp+e_{32}^\sharp ,\\
f_{22}=b_2e_{41}^\sharp+e_{42}^\sharp ,\\
f_{32}=(b_1+b_2^2)e_{11}^\sharp+b_2e_{12}^\sharp ,\\
f_{42}=(b_1+b_2^2)e_{21}^\sharp+b_2e_{22}^\sharp .
\end{cases}
\end{align}
We have $Xf_{42}=b_1 f_{11}+b_2f_{12}$.
Set $g(x)=-b_1 - b_2 x+x^2$. Then one verifies that
 $V^\sharp \simeq V(\om,w,g)$ via the map $f_{ks}\mapsto e_{ks}$.
 Since $g\sim f^\sharp$ we deduce that $V^\sharp\simeq V(\om,w,f^\sharp)$.
Thus, since polynomials in $\C[x]$ are similar iff they differ by a multiplicative scalar,
 $V\simeq V^\sharp$ iff $f=g$, i.e. iff $a_1=1/\overline{a_1}$ and $a_2=\overline{a_2}/\overline{a_1}$.
It is easy to check that
\[E:=\{(a_1,a_2)\in\C^2\; :\; a_1=1/\overline{a_1}, a_2=\overline{a_2}/\overline{a_1}\}
=\{(\zeta ^2, x\zeta) \; :\; x\in\mathbb{R}, \zeta\in\C, |\zeta |=1 \}\]
and
$(\zeta_1 ^2, x_1\zeta_1)=(\zeta_2 ^2, x_2\zeta_2)$ iff $(\zeta_1,x_1)=\pm (\zeta_2,x_2)$.

If $(a_1,a_2)\in E$, the non-degenerate admissible $\C$-form $\widehat{\Phi}$ corresponding
to the isomorphism $\Phi:V\to V^\sharp$, $\Phi(e_{ks})=f_{ks}$ is
\[\widehat{\Phi}(e_{ks},e_{lr})=\big(\Phi(e_{ks})\big)(e_{lr})=f_{ks}(e_{lr}).\]
Using \eqref{eq:ex_f38} and that $(e_{ks}^\sharp)(e_{lr})=\delta_{kl}\delta_{sr}$,
an explicit matrix for $\widehat{\Phi}$ in the basis $\{e_{ks}\}$ can be written down.
As a curious aside we mention that
the zero-set of the determinant of the symmetrized form
$\widehat{\Phi}+\widehat{\Phi}^\sharp$
as a function of $z\in\C\backslash\{1\}$ via $a_2=1-z$, $a_1=(1-z)/(1-\overline{z})$
is the curve known as the \emph{lima\c{c}on trisectrix}. It has certain special
geometric properties and is parametrized in polar coordinates by $r=1+2\cos\theta$.
Thus, for points outside of this curve, $\widehat{\Phi}+\widehat{\Phi}^\sharp$ is
the unique symmetric non-degenerate admissible form, by Remark \ref{rem:MTresults}.

\subsection*{Acknowledgements}
The author would like to thank L. Turowska for many interesting discussions and
helpful comments.

\end{document}